\documentclass[12pt]{amsart}

\usepackage{amsmath,amsthm,amsfonts,amssymb,bm,graphicx, mathrsfs}

\usepackage
{hyperref}
\hypersetup{colorlinks=true,citecolor=blue,linkcolor=blue,urlcolor=blue,
pdfstartview=FitH }

\theoremstyle{plain}
\newtheorem{theorem}{Theorem}

\newtheorem{lemma}{Lemma}

\numberwithin{equation}{section}

\theoremstyle{definition}

\renewcommand{\geq}{\geqslant}
\renewcommand{\leq}{\leqslant}
\renewcommand{\mod}{\mathrm{mod}\,}

\newcommand{\changed}[1]{{\color{black} #1}}

\usepackage{calc}
\newsavebox\CBox
\newcommand\hcancel[2][0.5pt]{%
  \changed{\ifmmode\sbox\CBox{$#2$}\else\sbox\CBox{#2}\fi%
  \makebox[0pt][l]{\usebox\CBox}%
  \rule[0.5\ht\CBox-#1/2]{\wd\CBox}{#1}}}

\makeatletter
\DeclareRobustCommand\widecheck[1]{{\mathpalette\@widecheck{#1}}}
\def\@widecheck#1#2{%
    \setbox\z@\hbox{\m@th$#1#2$}%
    \setbox\tw@\hbox{\m@th$#1%
       \widehat{%
          \vrule\@width\z@\@height\ht\z@
          \vrule\@height\z@\@width\wd\z@}$}%
    \dp\tw@-\ht\z@
    \@tempdima\ht\z@ \advance\@tempdima2\ht\tw@ \divide\@tempdima\thr@@
    \setbox\tw@\hbox{%
       \raise\@tempdima\hbox{\scalebox{1}[-1]{\lower\@tempdima\box
\tw@}}}%
    {\ooalign{\box\tw@ \cr \box\z@}}}
\makeatother

\begin{document}

\author{Valentin Blomer}

\address{Mathematisches Institut, Bunsenstr. 3-5, 37073 G\"ottingen, Germany} \email{vblomer@math.uni-goettingen.de}

\author{Xiaoqing Li}
\address{Department of Mathematics, 244 Mathematics Building, University at Buffalo, Buffalo, NY 14260-2900, USA}
\email{xl29@buffalo.edu}

\author{Stephen D. Miller}
\thanks{Li’s research was supported by NSF grant DMS-1303245.  Miller's research was supported by National Science Foundation grants DMS-1500562 and CNS-1526333.}
\address{Department of Mathematics, Hill Center,Busch Campus, Rutgers, 110 Frelinghuysen Rd., Piscataway, NJ 08854-8019, USA}
\email{miller@math.rutgers.edu}

 \title[Spectral reciprocity and nonvanshing for ${\rm GL}(4) \times {\rm GL}(2)$]{A Spectral reciprocity formula and non-vanishing for $L$-functions on ${\rm GL}(4) \times {\rm GL}(2)$}

\thanks{}


\begin{abstract}  We introduce a new type of summation formula for central  values of ${\rm GL}(4) \times {\rm GL}(2)$ $L$-functions, when varied  over Maa{\ss} forms.  By rewriting such a sum in terms of
${\rm GL}(4) \times {\rm GL}(1)$ $L$-functions and applying a new ``balanced'' Voronoi formula, the sum can be shown to be equal to a differently-weighted   average of the same quantities.  By controlling the support of the spectral weighting functions on both sides, this    reciprocity formula gives estimates on spectral sums that were previously obtainable only for lower rank groups.  The ``balanced'' Voronoi formula has Kloosterman sums on both sides, and can be thought of  as the
 functional equation of a certain double Dirichlet series involving Kloosterman sums and ${\rm GL}(4)$ Hecke eigenvalues.
  As an application we show that for any self-dual cusp form $\Pi$ for ${\rm SL}(4,\Bbb{Z})$, there exist infinitely many  Maa{\ss}  forms $\pi$ for ${\rm SL}(2,\Bbb{Z})$ such that $L(1/2, \Pi \times \pi) \not= 0$.
\end{abstract}


\setcounter{tocdepth}{2}  \maketitle

\maketitle

\section{Introduction}

This paper introduces a new method to investigate mean values of Rankin-Selberg $L$-functions. Our first main result is  an exact dualizing identity for sums of central  values of ${\rm GL}(4)\times {\rm GL}(2)$ $L$-functions (given in Theorem~\ref{thm3} below). It establishes a reciprocity phenomenon in the sense that families of $L$-functions associated with different and seemingly unrelated parts of the     ${\rm GL}(2)$-spectrum are put in an exact relation.

Sums of central $L$-values  are frequently used in subconvexity and nonvanishing theorems using techniques from analytic number theory.  These techniques were originally studied for $L$-functions related to automorphic forms on ${\rm GL}(2)$, with progress over the last fifteen years extending to ${\rm GL}(3)$.  However, all previous methods encounter barriers at ${\rm GL}(4)$.  One reason for this is the frequent use of Kuznetsov formulas to replace sums over ${\rm GL}(n)\times {\rm GL}(2)$ $L$-functions with sums over expressions more closely related to ${\rm GL}(n)\times {\rm GL}(1)$ $L$-functions.  For $n\le 3$ the Converse Theorem tells us that twists of $L$-functions of degree $n$ by ${\rm GL}(1)$ are sufficient (in principle) to determine all higher degree twists, whereas for $n=4$ twists by ${\rm GL}(2)$ $L$-functions are additionally required.

However, our formula goes against this philosophy and yields information about sums of ${\rm GL}(4)\times {\rm GL}(2)$ $L$-functions in terms of the simpler ${\rm GL}(4)\times {\rm GL}(1)$ $L$-functions.
Our main application (Theorem \ref{thm2} below) concerns the non-vanishing of ${\rm GL}(4)\times {\rm GL}(2)$ $L$-functions at the central point. More precisely, we answer in the affirmative the following question, for which ${\rm GL}(4)$ had previously been an analytic barrier:~given a  self-dual cuspidal automorphic representation  $\Pi$ on
${\rm GL}(4, \Bbb{Q})\backslash {\rm GL}(4, \Bbb{A})$, does there exist   a cuspidal automorphic representation $\pi$ on ${\rm GL}(2, \Bbb{Q}) \backslash {\rm GL}(2, \Bbb{A})$    such that $L(\frac 12, \Pi \times \pi) \not= 0$?
   For cohomological representations, strong methods based on $p$-adic interpolation (e.g.\ \cite{DJR}) are available to show non-vanishing of such twists.  In a different direction, approaches through period integrals   can sometimes show nonvanishing for ${\rm GL}(n) \times {\rm GL}(n-1)$ twists (see, for example, \cite{GJR,Harris}). Again this is mostly effective in cohomological situations.
For general automorphic forms it is mostly analytic techniques  that can be utilized, if at all. This is superficially similar to the Ramanujan conjecture, where there are many diverse tools   available for cohomological forms, but which   say nothing about the \emph{entire automorphic spectrum}. With purely analytic techniques, the ${\rm GL}(4)$-case is a well-known barrier, even in the  case of ${\rm GL}(1)$-twists considered in  \cite{Lu}.

A key tool in our reciprocity formula is a new Voronoi formula, described in Theorem~\ref{thm4} below. This formula   blends perfectly with the Kloosterman sums arising in the classical Kuznetsov formula, and leads to an interesting phenomenon:~the analysis of ${\rm GL}(4)\times {\rm GL}(2)$ $L$-functions is in this way reduced to an analysis of additive ${\rm GL}(1)$ twists of ${\rm GL}(4)$ Fourier coefficients (which is precisely  the typical use of  Voronoi formulas). In other words, in addition to reciprocity of ${\rm GL}(4)\times {\rm GL}(2)$ $L$-functions, there is also an exchange of information between ${\rm GL}(2)$ and ${\rm GL}(1)$ twists, and one can use lower rank twists to understand   higher rank twists. It is a fascinating question as to what extent this phenomenon can be transferred to more general groups.

Since this paper was first distributed, the ideas and methods have been used to develop  related spectral reciprocity  formulae  (see \cite{AK, BK1, BK2, Dj, Za}, which contain further applications). We now proceed to describe our results in detail.


\subsection{Spectral reciprocity identity}  Let $\Bbb{A}$ denote the adele ring of $\Bbb{Q}$ and fix   a  cuspidal automorphic representation $\Pi$ on ${\rm GL}(4, \Bbb{Q})\backslash {\rm GL}(4, \Bbb{A})$ (possibly, but not necessarily of cohomological type), which we assume to be unramified at all finite places.
Given a test function $h$ satisfying $h(t) \ll (1 + |t|)^{-A}$ for some $A > 4$, we define the following spectral mean value
\begin{equation}\label{defM}
\begin{split}
\mathcal{M}^{\pm}(h)  := \sum_{\pi} \epsilon_{\pi}^{(1\mp 1)/2} &\frac{L(1/2,   \Pi \times \pi ) }{L(1, \text{{\rm Sym}}^2 \pi)} h(t_{\pi})
\ + \\ & \frac{1}{2\pi} \int_{-\infty}^{\infty}  \frac{L(1/2 + it, \Pi)L(1/2 - it, \Pi)}{|\zeta(1 + 2 it)|^2} h(t) dt,
\end{split}
\end{equation}
where $\pi$ runs over all cuspidal automorphic representations of $\rm{GL}(2,\Bbb{A})$ corresponding to Maa{\ss} cusp forms   for ${\rm SL}_2(\Bbb{Z})$, and the parity $\epsilon_{\pi} \in \{\pm 1\}$ is the eigenvalue under the reflection operator $z \mapsto -\bar{z}$. In other words, the spectral average $\mathcal{M}^{-}(h)$ is twisted by the parity, while $\mathcal{M}^+(h)$ is not. The expression (\ref{defM}) is absolutely convergent by the convexity bounds (see \eqref{convex} below).  We also define a corresponding discrete series average
 \begin{equation}\label{defMhol}
\begin{split}
 \mathcal{M}^{\text{hol}}(h)    :=  \left.\sum_{\pi}\right.^{\text{hol}}   \frac{L(\frac 12,   \Pi \times \pi ) }{L(1, \text{{\rm Sym}}^2 \pi)} h(k_{\pi}),
\end{split}
\end{equation}
where the notation $\sum^{\text{hol}}$ indicates the sum is taken  over all  automorphic representations $\pi$ corresponding to classical holomorphic cusp forms of weight $k_{\pi}$  for ${\rm SL}(2,\Bbb{Z})$.
For $\diamondsuit \in \{+, -, \text{hol}\}$ we denote by $\widetilde{\mathcal{M}}^{\diamondsuit}$ the same expression as $ \mathcal{M}^{\diamondsuit}$, but with $\Pi$ replaced by its contragredient $\widetilde{\Pi}$ instead.
Let  $(\mu, \beta) \in \Bbb{C}^4 \times (\Bbb{Z}/2\Bbb{Z})^4$   denote the representation parameter of $\Pi$ as in \cite[\S1]{MS}, which  we assume (as we may after tensoring with a central character) satisfies
\begin{equation}\label{sumzero}
   \mu_1 + \mu_2+\mu_3 + \mu_4 = 0 \quad  \text{and} \quad   \beta_1 + \beta_2 + \beta_3 + \beta_4 \equiv 0 \  (\text{mod } 2)
 \end{equation}
 (see \text{\cite[(2.2)]{MS}}).  For example, if $\Pi$ is ${\rm SO}(4)$-fixed, then $\beta\equiv (0,0,0,0)$  or  $(1,1,1,1)$ (mod 2) depending on whether it corresponds to an even or odd automorphic form on ${\rm SL}(4,\Bbb{Z})\backslash {\rm SL}(4,\Bbb{R})/{\rm SO}(4)$.
Let
\begin{equation}\label{E01}
 \mathcal{E}_{0}(u)   :=    2 \cos\left(\frac{\pi}{2} u\right) \quad \text{and} \quad  \mathcal{E}_{1}(u)   :=  2i \sin\left(\frac{\pi}{2} u\right),
\end{equation}
where the index is understood as an element in $\Bbb{Z}/2\Bbb{Z}$, and define
\begin{equation}\label{defG}
\begin{split}
 \mathcal{G}^{\pm }(u)   :=   & \prod_{j=1}^4 \Gamma\left(u + \mu_j\right)  \Big( \mathcal{E}_{\beta_j}(u + \mu_j)    \mp     \mathcal{E}_{\beta_j+1}(u + \mu_j)\Big),
 \end{split}
 \end{equation}
 which are holomorphic in $\Re{u}>\frac 12-\delta$ for some constant $\delta$ depending only on $\Pi$ (see (\ref{deltarevised2})).
Moreover, let $\mathcal{C}^+(u, r) = \cos(\pi u/2)$, $\mathcal{C}^-(u, r) = \cosh(\pi r)$, and
\begin{equation}\label{kernel1}
 \begin{split}
 \mathcal{K}^{\pm}(t, r) :=  & \cosh(\pi t) \int_{\Re u=v}   \mathcal{G}^{\pm}\left(\frac{1-u}{2}\right) \mathcal{C}^{\pm}(u, r) \\
 &\times \Gamma(u/2 + it)\Gamma(u/2 - it)\Gamma(u/2 + ir)\Gamma(u/2 - ir)
   \frac{du}{2\pi i}.
   \end{split}
\end{equation}
The integrand is holomorphic in $0 < \Re u < 2\delta$ when $t,r\in\Bbb{R}$; the path of integration may be taken to be  $\Re u =v = \delta$.  We also define
\begin{equation}\label{kernel2}
\mathcal{K}^{\text{hol}}(t, k)    :=   \mathcal{K}^{+}\left(t, \frac{k-1}{2i}\right)
\end{equation}
for $k \in 2\Bbb{N}=\{2,4,6,\ldots\}.$
The integrals defining  $\mathcal{K}^{\diamondsuit}$ for $\diamondsuit \in \{+, -, \text{hol}\}$ are absolutely convergent for $t,r\in\Bbb{R}$ and $k\in 2\Bbb{N}$.

\begin{theorem}\label{thm3} Let $\Pi$ be a self-dual  cuspidal automorphic representation on ${\rm GL}(4,\Bbb{Q})\backslash {\rm GL}(4,\Bbb{A})$, which is unramified at all finite places.   Suppose that for some constant 
 $C_1 \geq 40$ the test function $h : \{t \in \Bbb{C} : |\Im t| < C_1 \} \rightarrow \Bbb{C}$
 \begin{equation}\label{c1}
 \begin{split}
&  \text{is holomorphic,  and satisfies
   $ h(t) \ll (1+|t|)^{-C_1}$}
   \\ & \text{
 and $h\left(\pm \left(n - \textstyle\frac{1}{2}\right) i\right) = 0$ for $n \in \Bbb{N}\cap[1,C_1].$}
\end{split}
 \end{equation}    Then
\begin{equation}\label{formula}
\mathcal{M}^-(h)  =   \widetilde{\mathcal{M}}^-(h^-) +  \widetilde{\mathcal{M}}^+(h^+) + \widetilde{\mathcal{M}}^{\text{{\rm hol}}}(h^{\text{{\rm hol}}}),
 \end{equation}
 where
 \begin{equation*}
  h^{\diamondsuit}(r)   :=    \frac{1}{8\pi^2} \int_{-\infty}^{\infty}  \mathcal{K}^{\diamondsuit}(t, r) h(t) t \tanh(\pi t) \frac{dt}{2\pi^2}
 \end{equation*}
is absolutely convergent  for each $\diamondsuit \in \{+, -, \text{hol}\}$. Moreover, the spectral sums $\widetilde{\mathcal{M}}^-(h^-)$, $\widetilde{\mathcal{M}}^+(h^+)$, and $\widetilde{\mathcal{M}}^{\text{{\rm hol}}}(h^{\text{{\rm hol}}})$   are each absolutely convergent, and
 \begin{equation}\label{small}
\mathcal{K}^{+}(t, r)    \ll    e^{-\frac{\pi}{2} |r|} (1 + |t|)^{-1}
 \end{equation}
 for $t, r \in \Bbb{R}$.
  \end{theorem}

The integral kernels  $\mathcal{K}^{\diamondsuit}(t, r)$  will be further discussed  in Section \ref{interlude}, where they are computed explicitly  in terms of $_4F_3$ hypergeometric functions. 
%
In practical situations, the terms $\widetilde{\mathcal{M}}^{+}(h^+)$ and $\widetilde{\mathcal{M}}^{\text{hol}}(h^{\text{hol}})$ are easy to bound and/or small, so that we  have a ``reciprocity formula''
\begin{equation}\label{reci}
\mathcal{M}^-(h)  \rightsquigarrow    \widetilde{\mathcal{M}}^-(h^{-}).
\end{equation}
An analysis of the $T$-dependence   shows that if $h(t)$ is (essentially) supported on $t \ll T$, then $h^{-}(r)$ is of size $T$ for  $r \ll 1$ and very small otherwise. Refining this argument, if $h(t)$ is (essentially) supported on a short segment $(T-M,T+M)$ for $T^{1/3} \leq M \leq T$, then $h^{-}(r)$ is essentially supported on  $r \ll T/M$ and very small otherwise. This is a typical ``duality''  phenomenon of automorphic summation formulae.

Formula (\ref{formula}) is motivated by earlier work of Kuznetsov and Motohashi.
If $\Pi$ is replaced by a minimal parabolic Eisenstein series, the spectral mean value $\mathcal{M}^{\pm}(h)$ features a fourth moment of ${\rm GL}(2)$  $L$-functions   in the cuspidal sum, along with the eighth moment of the Riemann zeta-function in the contribution of the continuous spectrum. A similar reciprocity formula to (\ref{formula}) in this case was   envisaged by Kuznetsov, and completed by Motohashi \cite{Mo3}. As far as we know, this
interesting formula has not yet been used in applications.

 \subsection{Balanced Voronoi summation}

One of the tools in the proof of Theorem \ref{thm3} is the analytic continuation and  functional equation of a certain  double Dirichlet series, which may be of independent interest. Let $a_{\Pi}(n_1, n_2, n_3)=a_{\Pi}(|n_1|, |n_2|, |n_3|)$ denote the abelian  coefficients of $\Pi$ (see Section~\ref{sec2}), normalized such that $a_{\Pi}(1, 1, 1) = 1$; they are the Hecke eigenvalues of $\Pi$ when $n_1,n_2,n_3>0$. Recall that 
$a_{\widetilde{\Pi}}(n_1, n_2, n_3) = a_{\Pi}(n_3, n_2, n_1) = \overline{a_{\Pi}(n_1, n_2, n_3)}$.  For $\epsilon \in \Bbb{Z}/2\Bbb{Z}$ denote
\begin{equation}\label{defV}
\mathcal{V}_{\epsilon}(s, z)  :=   \sum_{n \in \Bbb{Z} \setminus \{0\}} \sum_{m, c > 0} \frac{ \text{sgn}(n)^{\epsilon} a_{\Pi}(n, m, 1)}{(|n|m^2)^z} \frac{S(n, 1, c)}{c} \left(\frac{|n|}{c^2}\right)^{-s},
\end{equation}
and write $\widetilde{\mathcal{V}}_{\epsilon}$ for the same expression with $a_{\widetilde{\Pi}}(n, m, 1) = a_{\Pi}(1, m, n)$ instead of $a_{\Pi}( n, m, 1)$.
The multiple sum is absolutely convergent in the range
$$\Re s < -1/4, \quad \Re(s+z) > 1,$$
as follows from the Weil bound on Kloosterman sums and the estimate (\ref{RSbound}) below.
The following result is a consequence of   the ``balanced'' Voronoi formula given in Theorem~\ref{thm:balanced}, which is itself a consequence of the usual GL(4) Voronoi formula from \cite{MS} (though packaged in a very different way).

 \begin{theorem}\label{thm4} The function $\mathcal{V}_{ \epsilon}(s, z)$ has holomorphic  continuation to the region  $$\{(s,z)\in\Bbb{C}^2  \mid  \Re (s + 2z) > 5/4,  \Re z > 5/4, \,  \text{{\rm and }} \Re{s}<-1/4\}$$ and satisfies the functional equation
 \begin{equation}\label{voronoiFE}
   \mathcal{V}_{\epsilon}(s, z)    =   \mathcal{G}_{\epsilon}(1 - s - z)\widetilde{\mathcal{V}}_{ \epsilon}(1-s-2z, z),
 \end{equation}
 where
 \begin{equation}\label{curlyGeps}
 \mathcal{G}_{\epsilon}(s)    :=    \prod_{j=1}^4 \frac{\Gamma(s+\mu_j)}{(2\pi)^{s+\mu_j} } \mathcal{E}_{\epsilon + \beta_j}(s+\mu_j)
 \end{equation}
with $\mathcal{E}$ as in \eqref{E01}.
 \end{theorem}

\noindent
Notice that
 \begin{equation}\label{similar}
    \mathcal{G}^{\pm}(u)   =  (2\pi)^{4u} (\mathcal{G}_0(u) \mp \mathcal{G}_1(u))
 \end{equation}
because of \eqref{sumzero} and \eqref{defG}.
When integrated against the Mellin transform of a test function, (\ref{voronoiFE}) assumes the symmetric, equivalent form of a  summation formula for   sums of abelian  coefficients times Kloosterman sums.  This formula was first obtained by the second two named authors, and has since been generalized to arbitrary sums of GL($n$) Fourier coefficients weighted by hyper-Kloosterman sums in \cite{Zh,MZ}.

 \subsection{Application:~non-vanishing of $L$-functions} Let $1 \leq m < n$ be two positive integers.  Given a  cuspidal automorphic representation  $\Pi$ on ${\rm GL}(n, \Bbb{Q})\backslash {\rm GL}(n, \Bbb{A})$, does there exist   a self-dual cuspidal automorphic representation $\pi$ on ${\rm GL}(m, \Bbb{Q}) \backslash {\rm GL}(m, \Bbb{A})$   such that $L(\frac 12, \Pi \times \pi) \not= 0$?    The case $m=1$ (i.e.,\ $\pi$ corresponds to a quadratic Dirichlet character) is one of the prime applications of the theory of multiple Dirichlet series, through which it is possible to answer this question affirmatively for  $n \in \{2, 3\}$ (see \cite{Bu,HK}). The case $n=3$ and $m=2$ follows  from  a result of the second named author  \cite[Theorem 1.1]{Li}

For $n \geq 4$ and any value of $m$, the problem becomes very hard and is completely open if one is interested in results that hold across the entire spectrum and are not restricted to cohomological representations.  This is not particularly surprising in view of lack of progress on related problems (such as subconvexity and counts for zeros on the critical line), which remain open for cuspidal automorphic $L$-functions on ${\rm GL}(n)$, $n \geq 4$. Indeed, even already on ${\rm GL}(3)$ the state-of-the-art analytic machinery (e.g., spectral summation formulae, multiple Dirichlet series, and period formulae) turns out to be unsuccessful for many problems which have long been settled for ${\rm GL}(2)$.

In this article we solve the non-vanishing problem in the case $(n, m) = (4, 2)$ when  $\Pi$ is self-dual and unramified at all finite places  (i.e.,  $\Pi$ is associated to a cusp form on ${\rm GL}(4,\Bbb{Z})\backslash {\rm GL}(4,\Bbb{R})$). The most direct approach, i.e.\ asymptotically evaluating a quantity of the shape
\begin{equation*}
\sum_{\pi} e^{-(t_{\pi}/T )^2} \frac{L(1/2, \Pi \times \pi)}{L(1, \text{Sym}^2 \pi)},
\end{equation*}
fails, at least with currently available tools. It quickly leads to a ``deadlock'', as we shall describe at the end of the introduction in more detail.  This phenomenon is well-known to experts and is a serious in many applications of analytic number to automorphic forms.  Therefore we introduce a different path which we now proceed to describe.

Our nonvanishing result will be shown as a consequence of Theorem~\ref{thm3}, more precisely a  quantitative version that comes as a corollary of it.  Let $D \geq 50$ be  a   fixed  integer and define
\begin{equation}\label{defH}
h_T(t) := e^{-(t/T)^2} P_T(t), \quad  \text{where} \quad  P_T(t) := \!\left(\prod_{n=1}^D  t^2 + \left(\frac{2n-1}{2}\right)^2\right)^2\!\! T^{-4D}
\end{equation}
for  a   parameter $T > 1$ (that will be taken large). 

\begin{theorem}\label{thm1} Let $\Pi$ be a cuspidal automorphic representation of ${\rm GL}(4, \Bbb{A})$ which is unramified at all finite places, and let   $h_T$ be as in \eqref{defH} 
for  some fixed integer $D \geq 50$. Then 
\begin{equation}\label{mainbound}
\mathcal{M}^-(h_T) \ll T \quad  \text{for} \quad T > 1,
\end{equation}
where the implied constant depends only on $\Pi$  and $D$.
\end{theorem}
Henceforth we shall fix the cusp form $\Pi$ and not display the implicit dependence of any constants on it.
If the generalized Lindel\"of hypothesis is true, then by Weyl's law the  $\pi$-sum in \eqref{defM} has roughly $T^2$ terms that are essentially bounded, so   Theorem \ref{thm1} achieves square root cancellation  and its linear exponent of $T$ is expected to be  best possible. Notice that \eqref{mainbound} is a \emph{pure bound that contains no $\varepsilon$ in the exponent}, which is absolutely crucial for our application. 

 Of course, due to varying signs Theorem \ref{thm1} cannot be used to bound the individual terms in the sum \eqref{defM}.  Nevertheless, if $\Pi$ is self-dual we can still solve the non-vanishing problem  mentioned at the beginning of the paper for $(n,m)=(4,2)$. In this case we have
\begin{equation*}
  L(\textstyle{\frac 12} + it, \Pi)L(\textstyle{\frac 12} - it, \Pi) = L(\textstyle{\frac 12} + it, \Pi)L(\textstyle{\frac 12} - it, \widetilde{\Pi}) = |L(\textstyle{\frac 12} + it, \Pi)|^2,
\end{equation*}
where $\widetilde{\Pi}$ denotes the contragredient of $\Pi$.
It follows from the lower bound technique of Rudnick and Soundararajan \cite{RS} that the Eisenstein term in \eqref{defM} is
\begin{equation}\label{lower}
\frac{1}{2\pi} \int_{-\infty}^{\infty}  \frac{ |L(1/2 + it, \Pi)|^2}{|\zeta(1 + 2 it)|^2} h_T(t) dt    \gg   T \log T
\end{equation}
for $h_T$ as in \eqref{defH} and $T$ sufficiently large, as we show in  Lemma~\ref{lem:lowerbound}. Comparing with \eqref{mainbound}, we immediately conclude  that not all terms in the $\pi$-sum in \eqref{defM} can vanish  and, quite surprisingly,  there must be a small bias towards odd cusp forms:~the (conjecturally) non-negative values $L(\frac 12, \Pi \times \pi)$ are typically a little bit larger for odd cusp forms $\pi$ than for even $\pi$,   to counteract the contribution of the (always even) Eisenstein series. This yields the main application of the reciprocity formula \ref{thm3}, and to our knowledge is the first time that a non-vanishing result is deduced from an \emph{oscillating} mean value result.

\begin{theorem}\label{thm2} Let $\Pi$ be a self-dual  cuspidal automorphic representation on ${\rm GL}(4,\Bbb{Q})\backslash {\rm GL}(4,\Bbb{A})$ which is unramified at all finite places.  Then there exist infinitely many   cuspidal automorphic representations $\pi$  associated to odd  Maa{\ss} cusp forms for ${\rm SL}(2, \Bbb{Z})$   such that $L(\frac 12,   \Pi \times \pi) \not= 0$.

Quantitatively, for every constant $c_0 \geq 1$ there exists $T_0 = T_0(\Pi, c_0) > 10$  such that for every $T \geq T_0$  at least $c_0 \log\log T$ representations $\pi$ with spectral parameter $t_{\pi} \leq T$ satisfy $L(1/2,   \Pi \times \pi) \not= 0$.
\end{theorem}

There are various different types of self-dual cuspidal automorphic representations $\Pi$ on ${\rm GL}(4)$.  For example, $\Pi$ can be a Rankin-Selberg lift from ${\rm GL}(2)\times {\rm GL}(2)$, a symmetric cube lift from ${\rm GL}(2)$, or more generally a lift from ${\rm O}(3,2)$ or ${\rm Sp}(4)$.  The non-negativity of $L(1/2,\Pi\times \pi)$ is known when $\Pi$ is of orthogonal type \cite{lapid}.

\subsection{Heuristics and concluding remarks} It is instructive to see a back-of-the-envelope computation that leads to \eqref{mainbound}. The analysis is based on the well-proven  Kuznetsov-Voronoi-Kuznetsov triad:~starting with an approximate functional equation, we need to bound an expression very roughly of the form
\begin{equation}\label{negative1}
\sum_{t_{\pi} \ll T}  \epsilon_{\pi}  \sum_{n \ll T^4} \frac{\lambda_{\Pi}(n) \lambda_{\pi}(n)}{n^{1/2}}   =  \sum_{t_{\pi} \ll T}    \sum_{n \ll T^4} \frac{\lambda_{\Pi}(n) \lambda_{\pi}(-n)}{n^{1/2}},
\end{equation}
where $\lambda_{\Pi}$ and $\lambda_{\pi}$ denote the Hecke eigenvalues of $\Pi$ and $\pi$ respectively.
(Strictly speaking, this sum along with the others in this paragraph should be taken with smooth cutoffs.)
 Using the   ``opposite-sign'' Kuznetsov formula gives an expression involving a sum over Kloosterman sums, which is very roughly of the shape
\begin{equation*}\label{negative2}
T \sum_{n \ll T^4} \frac{\lambda_{\Pi}(n)}{n^{1/2}} \sum_{c \ll T} \frac{ S(-n, 1, c)}{c}.
\end{equation*}
  There are no additional  oscillatory terms in this sum because the integral kernel (\ref{ker}) features the $K$-Bessel function (cf.\ \cite[Lemma 3.8]{BK}, for example).
This is followed by an application of the Voronoi summation formula on ${\rm GL}(4)$. The dual $n$-sum will then be essentially of bounded length, and
 since $2-4 = -2$,  the ${\rm GL}(2)$ Kloosterman sum $S(-n, 1, c)$ remains essentially invariant under the ${\rm GL}(4)$ Voronoi formula.  We make this precise in Theorem \ref{thm4} below  in terms of a ``balanced'' Voronoi formula that includes   Kloosterman sums on both sides.  One  obtains from this an expression  roughly of the form
 \begin{equation}\label{negative3}
 T \sum_{n \ll 1} \lambda_{\Pi}(n)  \sum_{c \ll T} \frac{1}{c} S(-n, 1, c),
 \end{equation}
 after which  the Kuznetsov formula is applied again -- but in reverse.  This sequence is not involutory and yields full square-root cancellation in the $c$-sum; it therefore bounds \eqref{negative3} by $O(T)$, at least if there are no exceptional eigenvalues (which is known for the full-level modular group ${\rm SL}(2,{\Bbb Z})$).

 This   heuristic reasoning described here is of course insensitive to $\varepsilon$-powers.  In order to delicately obtain the $O(T)$ estimate in Theorem~\ref{thm1} we will use   a more structural approach to studying the spectral mean value \eqref{defM}. In addition, it is crucial for Theorem~\ref{thm2} that (\ref{defV}) includes the precise Kloosterman sum arising from the Kuznetsov formula.  Here one cannot simply open up the Kloosterman sum and apply Voronoi summation in its usual form \cite{MS} as in the subconvexity results of \cite{Sarnak,Lisubc}, because this entails the loss of a multiplicative term of size $T^{\varepsilon}$ owing to appearance of   additional divisor  sums.  Rather, it is a special feature of the GL(4) Hecke algebra that produces the precise form of (\ref{voronoiFE}), and allows us to cleanly avoid those extra factors.  Without this the $O(T)$ bound in (\ref{mainbound}) would instead exceed the main term $T\log T$, and we would not able to deduce  our nonvanishing result.

 We conclude the introduction with a discussion of whether there is a reciprocity formula for  the untwisted spectral average $\mathcal{M}^+(h)$ in analogy to \eqref{reci}.   Were we to apply the heuristic analysis of \eqref{negative1}--\eqref{negative3} without  $\epsilon_{\pi}$, the ``same-sign'' Kuznetsov formula 
 produces roughly
$$T \sum_{n \ll T^4} \frac{\lambda_{\Pi}(n)}{n^{1/2}}e(\pm 2 \sqrt{n}), \quad \text{where}  \quad  e(x) :=  e^{2\pi i x}, $$
instead of (\ref{negative2}),
since   the $c$-sum is now essentially bounded.  However, this comes at the cost of the  archimedean phase factor $e(\pm 2 \sqrt{n})$ from the Bessel function.    At this point the Voronoi summation formula applied to the $n$-sum is completely self-dual, and so the triad Kuznetsov-Voronoi-Kuznetsov turns out to be essentially  involutory  and gives no useful information. This is the ``deadlock'' situation mentioned earlier.  The extra oscillation introduced by the parity breaks the self-duality here,  in such a strong sense that one can obtain pure bounds that are precise even on a $\log$-scale, as a comparison of \eqref{mainbound} and \eqref{lower} demonstrates.

As a reflection of this phenomenon,
formally imitating  the proof of  Theorem \ref{thm3} -- but without the parity --  yields serious convergence problems:~even the final spectral formula does not converge. This was already observed by Motohashi in the case of Eisenstein series \cite[(2.16)]{Mo3}. The problem can be repaired   by  restricting to test functions $h$ with spectral mass 0, or more precisely by  requiring that
\begin{equation*}
\int_{-\infty}^{\infty} h(t) t^n \tanh(\pi t) dt = 0
\end{equation*}
for all integers $n$ up to some sufficiently large bound. 
Alternatively, one can   consider the difference $\mathcal{M}^+(h) - \mathcal{M}^{\text{hol}}(h)$, for which the convergence problems disappear and one can derive a spectral identity.  These devices are reminiscent of \cite[Section 3]{Mo3}. As we are not aware   of any   interesting applications of such formulae,   we shall not go into further detail here. \\

{\bf Acknowledgements:} The authors wish to thank Michael Harris, Peter Sarnak,  Wilfried Schmid, Matthew Young, and Fan Zhou  for their helpful comments and suggestions.

 \section{Automorphic toolbox}\label{sec2}

For the rest of the paper $\pi$ will denote a cuspidal  automorphic representation on ${\rm GL}(2, \Bbb{Q})\backslash {\rm GL}(2, \Bbb{A})$ associated to a    Hecke-Maa{\ss} cusp form for ${\rm SL}(2, \Bbb{Z})$.  Let $t_\pi\geq 9.7$  denote the spectral parameter of $\pi$ (see  \cite[Appendix C]{He}) and  let $\lambda_{\pi}(n)$ denote its Fourier coefficients, normalized so that $\lambda_\pi(1)=1$.  The parity is $\epsilon_{\pi}=+1$ for even Maa{\ss} forms and $-1$ for odd Maa{\ss} forms, so that $\lambda_{\pi}(-n) = \epsilon_{\pi} \lambda_{\pi}(n)$.
  The Fourier coefficients of the non-holomorphic Eisenstein series $E_{it}$, $t \in \Bbb{R}$, are given explicitly as divisor sums
 $$\lambda_t(n)  :=   \sum_{0<d\mid n} d^{it} (n/d)^{-it}$$
for $n\not= 0$.  Let
 \begin{equation}\label{dspect}
  d_{\text{spec}}t   :=   \frac{1}{2\pi^2} t \tanh(\pi t)  dt
 \end{equation}
denote the Plancherel measure for  ${\rm SL}(2,\Bbb{R})$.

 As in the introduction, let   $\Pi$ be a cuspidal automorphic representation on ${\rm GL}(4,\Bbb{Q})\backslash {\rm GL}(4,\Bbb{A})$ which is unramified at all finite places.  Its abelian coefficients $a_\Pi(n_1,n_2,n_3)$ are the Hecke eigenvalues of $\Pi$ when each $n_i>0$, and by definition satisfy $a_\Pi(\sigma_1 n_1,\sigma_2 n_2,\sigma_3 n_3)=a_\Pi(n_1,n_2,n_3)$ for any choice of $\sigma_i=\pm 1$.  They are related to those of $\widetilde\Pi$ by
 $$a_{\widetilde{\Pi}}(n_1, n_2, n_3)   =   a_{\Pi}(n_3, n_2, n_1)   =  \overline{a_{\Pi}(n_1, n_2, n_3)} =   a_{\Pi}(-n_3, n_2, n_1).$$  The following bound of   Serre  (reprinted in \cite[appendix]{BB})  holds uniformly for all finite places:~there exists $\delta > 0$ such that
  \begin{equation}\label{delta1}
 \lambda_{\Pi}(n)   := a_{\Pi}(n, 1, 1)   \ll   n^{1/2 - \delta}.
 \end{equation}
 We will frequently use the Rankin-Selberg bound
 \begin{equation}\label{RSbound}
  \sum_{nm^2 \leq x} |a_{\Pi}(n, m, 1)|^2\ll x,
 \end{equation}
typically in combination with the Cauchy-Schwarz inequality; it is a consequence of the fact that the degree-16 tensor product $L$-function $L(s, \Pi \times \widetilde{\Pi})$ has a simple pole at $s=1$ and analytic continuation with moderate growth in vertical strips.  It is also known \cite[Proposition 6.2]{Ki} that the  Euler product for the  degree-10 $L$-function $L(s, \Pi, \text{Sym}^2)$ is absolutely convergent for $\Re s > 1$.   With this background in hand, we now turn to the following lemma (whose   proof is, in absence of the Ramanujan conjecture, not completely trivial).

\begin{lemma}\label{lem0} For $x \geq 1$ we have
$$\sum_{n \leq x} |\lambda_{\Pi}(n)|^2   \gg    x,$$
where the implied constant depends on $\Pi$.
\end{lemma}

\begin{proof} 
 We start with a variation of \cite[Proposition 2.4]{RSa} for $m=4$. For a prime $p$ let $\{\alpha_j(p)\mid j = 1, \ldots, 4\}$, denote the Satake parameters of $\Pi$ at $p$, so that $$\lambda_{\Pi}(p)  =  \sum_{j=1}^4 \alpha_j(p), \quad \lambda_{\Pi}(p^2)   =   \sum_{1 \leq i \leq j \leq 4} \alpha_i(p) \alpha_j(p),$$
 and $\lambda_\Pi(p^k)$ is a symmetric polynomial in the $\alpha_j(p)$ of degree $k$.
These parameters are normalized to satisfy  $\alpha_1(p)  \cdots  \alpha_4(p) = 1$  and  $\{\alpha_j(p)^{-1}\mid  j = 1, \ldots, 4\} = \{\overline{\alpha_j(p)}\mid j = 1, \ldots, 4\}$,.  Therefore there are three possibilities for their size: (a) the Ramanujan conjecture holds at $p$, i.e., $|\alpha_j(p)| = 1$ for $1 \leq j \leq 4$; (b) two of the parameters, say $\alpha_1(p)$, $\alpha_2(p)$,  are on the unit circle, while $\alpha_3(p)^{-1} = \overline{\alpha_4(p)}$ is not; or (c) $\{\alpha_1(p), \alpha_2(p), \alpha_3(p), \alpha_4(p)\} = \{\alpha, \bar{\alpha}^{-1}, \beta, \bar{\beta}^{-1}\}$ (as multisets) for complex numbers $\alpha, \beta$ not on the unit circle.
In each case it is easy to see that
$$\max_{j} |\alpha_j(p)|^2   \ll   1 + |\lambda_{\Pi}(p)|^2  +  | \lambda_{\Pi}(p^2)|.$$
Let $a_\Pi$ denote the multiplicative function defined on prime powers $p^k$ by
$a_{\Pi}(p^k) =   \alpha_1(p)^k + \cdots + \alpha_4(p)^k $
(which coincides with $\lambda_{\Pi}$ on primes); it thus satisfies the bound
$$|a_{\Pi}(p^k)|^2  \ll_k   1  +   |\lambda_{\Pi}(p)|^{2k}    +  | \lambda_{\Pi}(p^2)|^k$$
for $k \in \Bbb{N}$.

The quantity $|\lambda_\Pi(p)|^2$ is the Dirichlet series coefficient of $p^{-s}$ in the degree-16 Rankin-Selberg $L$-function $L(s, \Pi \times \widetilde{\Pi})$, while the quantity $\lambda_\Pi(p^2)$ is the  Dirichlet series coefficient of $p^{-s}$ in the degree-10 symmetric square $L$-function $L(s, \Pi, \text{Sym}^2)$.
As noted above, both $L$-functions are uniformly and absolutely convergent in the half plane $\Re{s}\geq 1+\varepsilon$, for any $\varepsilon>0$, and consequently
$\sum_{p\leq x}|\lambda_\Pi(p)|^2$ and  $\sum_{p\leq x}|\lambda_\Pi(p^2)|$ are both $O_\varepsilon(x^{1+\varepsilon})$.
It follows that
\begin{equation}\label{k1}
\begin{split}
\sum_{p \leq x} |a_{\Pi}(p^k)|^2 &  \ll_k   \sum_{p \leq x} \left(1 + |\lambda_{\Pi}(p)|^{2k} + |\lambda_{\Pi}(p^2)|^{k}\right)\\
&   \ll_{\varepsilon}   x + x^{1+\varepsilon} x^{(\frac{1}{2} - \delta)(2k-2)} + x^{1+\varepsilon} x^{(1 - 2\delta)(k-1)} \ll x^{k - 2\delta(k-1) + \varepsilon}
\end{split}
\end{equation}
for any $\varepsilon > 0$, where we have applied the bound
 \eqref{delta1}
 to $ |\lambda_{\Pi}(p)|^{2k-2}$ and $|\lambda_{\Pi}(p^2)|^{k-1}$.

The implied constant in  (\ref{k1}) depends on $k$, but for
 $k \geq 1/\delta$ the bound \eqref{delta1} trivially implies
\begin{equation}\label{k2}
\sum_{p \leq x} |a_{\Pi}(p^k)|^2  \ll    \sum_{p \leq x} p^{2k(\frac{1}{2} - \delta)}  \leq   x^{1 + k - 2k\delta}.
\end{equation}
Combining \eqref{k1} and \eqref{k2} we deduce
\begin{equation}\label{k3}
\sum_{p^k \leq x} |a_{\Pi}(p^k)|^2   \ll_{\varepsilon}  x^{1 - \delta+\varepsilon}
\leq
 x^{1-\delta/2}
, \quad \text{uniformly for }  k\ge2
\end{equation}
and $\varepsilon < \delta$.  By the ``prime number theorem'' for  Rankin-Selberg $L$-functions \cite[Lemma 5.1]{LWY} we have
\begin{equation*}
\sum_{n \leq x} \Lambda(n) |a_{\Pi}(n)|^2   \sim   x \qquad \text{as~~}x\rightarrow\infty,
\end{equation*}
where the von Mangoldt function $\Lambda(n)$ is supported on prime powers and $\Lambda(p^k)=\log(p)$.
By \eqref{k3} the contribution from
   higher prime powers is negligible in this sum, and we conclude
 \begin{equation}\label{PNT}
 \sum_{p \leq x}   |\lambda_{\Pi}(p)|^2   =   \sum_{p \leq x}   |a_{\Pi}(p)|^2    \sim   \frac{x}{\log x} \quad \text{as} \quad x\rightarrow\infty
 \end{equation}
  by partial summation.

 We now apply  Wirsing's Theorem \cite[Satz 1]{Wi} to the multiplicative function
  $$f(n)= \left\{
            \begin{array}{ll}
               |\lambda_{\Pi}(n)|^2, &  n  \hbox{~squarefree;} \\
              0, & \hbox{otherwise}
            \end{array}
          \right.
$$ to deduce for large $x$ that
\begin{displaymath}
\begin{split}
\sum_{n \leq x} f(n)   \sim  e^{-\gamma} \frac{x}{\log x} \prod_{p \leq x}& \left(1 +  \frac{f(p)}{p}\right) \geq
\\
& e^{-\gamma} \frac{x}{\log x} \exp\left(\sum_{p \leq x} \frac{|\lambda_{\Pi}(p)|^2}{p} -\frac{1}{2}\sum_{p \leq x} \frac{|\lambda_{\Pi}(p)|^4}{p^2} \right).
\end{split}
\end{displaymath}
Here we   used the inequality $\log (1+x) \geq x - x^2/2$ for $x \geq 0$.
By applying \eqref{PNT} and \eqref{delta1}, along with partial summation  and the absolute convergence of $L(s,\Pi\times\widetilde \Pi)$, we have
 \begin{equation*}
 \begin{split}
 \sum_{p \leq x} \frac{|\lambda_{\Pi}(p)|^2}{p}   -  \frac{1}{2}\sum_{p \leq x} \frac{|\lambda_{\Pi}(p)|^4}{p^2} & = \log\log x + O(1) + O\Bigl(\sum_{p \leq x} \frac{|\lambda_{\Pi}(p)|^2}{ p^{1+2\delta}} \Bigr) \\ & = \log\log x + O(1),
 \end{split}
 \end{equation*}
so that
$$\sum_{n\leq x}    |\lambda_\Pi(n)|^2   \geq   \sum_{n\leq x}f(n)   \gg   x$$
as asserted.
\end{proof}

  It follows from   the Dirichlet series expansions of tensor product $L$-functions in \cite[\S2]{BumpOslo}  that
\begin{equation}\label{diri1}
L(s + it, \Pi)L(s - it, \Pi)  =   \sum_{n, m\ge1} \frac{\lambda_{t}(n) a_{\Pi}(n, m, 1)}{n^sm^{2s}}
\end{equation}
and
\begin{equation}\label{diri2}
L(s, \Pi \times \pi)   =   \sum_{n, m\ge1} \frac{\lambda_{\pi}(n) a_{\Pi}(n, m, 1)}{n^sm^{2s}},
\end{equation}
 both of which   converge absolutely  for $\Re s > 1$ \cite[Theorem~5.3]{JS}.
In particular, all
 $L$-functions appearing in \eqref{defM} have Dirichlet series expansions which are absolutely convergent in $\Re s > 1$.
The convexity bound for central $L$-values and lower bounds for $L$-functions at the edge of the critical strip \cite{HL}, \cite[Section 3.6]{Ti} imply the estimates
\begin{equation}\label{convex}
\begin{split}
\frac{L(1/2 + it, \Pi)L(1/2 - it, \Pi)}{|\zeta(1 + 2 it)|^2}   & \ll_\varepsilon  (1 + |t|)^{2+\varepsilon} \min(1, |t|^2)
\\  \text{and} \quad
\frac{L(1/2, \Pi \times \pi)}{L(1, \text{Sym}^2 \pi)} & \ll_\varepsilon
\left\{
  \begin{array}{ll}
    t_{\pi}^{2+\varepsilon} \\
    k_{\pi}^{2+\varepsilon},
  \end{array}
\right.
\end{split}
\end{equation}
 where the extra factor of $t^2$ for $t$ small comes from the pole of the Riemann $\zeta$-function and the two cases on the right hand side correspond to Maa\ss~forms and holomorphic modular forms, respectively.

Our application of Theorem~\ref{thm4} requires some information about the representation parameter $(\mu,\beta)\in \Bbb{C}^4\times (\Bbb{Z}/2\Bbb{Z})^4$   appearing in  (\ref{curlyGeps}).  For  $\eta\in \Bbb{Z}/2\Bbb{Z}$ let
\begin{equation}\label{G01def}
  G_\eta(s)  :=  (2\pi)^{-s}\Gamma(s){\mathcal E}_{\eta}(s)  =
  \left\{
    \begin{array}{ll}
      \displaystyle  \frac{\Gamma_{\Bbb{R}}(s)}{\Gamma_{\Bbb{R}}(1-s)}, & \eta\in 2\Bbb{Z}, \\[0.5cm]
      \displaystyle  i \frac{\Gamma_{\Bbb{R}}(s+1)}{\Gamma_{\Bbb{R}}(2-s)}, &  \eta\in 2\Bbb{Z}+1,
    \end{array}
  \right.
\end{equation}
 where $\Gamma_{\Bbb{R}}(s)=\pi^{-s/2}\Gamma(s/2)$.  Thus ${\mathcal G}_\epsilon(s)=\prod_{j\leq 4}G_{\epsilon+\beta_j}(s+\mu_j)$ in (\ref{curlyGeps}).  Representation parameters for all cusp forms on ${\rm GL}(n,\Bbb{R})$ are explicitly described in \cite[(A.1)-(A.2)]{mirabolic}.  In particular, each $\mu_j$ either has the form $s'$ or occurs in a pair $(\mu_j,\mu_{j'})=(s'-\frac{k-1}{2},s'+\frac{k-1}{2})$  for some   $s'\in\Bbb{C}$ with $|\Re{s'}|<\frac 12$ and $k\in \Bbb{Z}_{\geq 2}$ with $k\equiv \beta_j+\beta_{j'}~ (\mod 2)$.  Although in the latter situation
  $G_{\epsilon+\beta_{j}}(s+\mu_{j})$ has simple poles at $s=\frac{k-1}{2}-s'-n$, $n\in \Bbb{Z}_{\geq 0}$, the product
\begin{equation}\label{Gproductscancel}
  G_{\epsilon+\beta_j}(s+\mu_j) G_{\epsilon+\beta_{j'}}(s+\mu_{j'})=
  i^k(2\pi)^{1-2s-2s'}\frac{\Gamma(s+\frac{k-1}{2}+s')}{\Gamma(1-s+\frac{k-1}{2}-s')}
\end{equation}
is holomorphic in $\{s\mid \Re{s}>0\}$.  Thus
\begin{equation}\label{deltarevised2}
  {\mathcal G}_0(s), \ \ {\mathcal G}_1(s), \ \ \text{and} \ \ {\mathcal G}^\pm(s) \quad \text{are all holomorphic in} \ \Re{s}>\frac 12-\delta
\end{equation}
for some sufficiently small $0<\delta<1/2$ depending on $\Pi$ (which we assume, as we may, is simultaneously valid  in (\ref{delta1})).
Stirling's formula applied to (\ref{G01def}) gives the asymptotics   $|G_\eta(s)|\sim |\frac{s}{2\pi}|^{\Re{s}-1/2}$ in vertical strips of finite width,  hence  using (\ref{sumzero}) we bound
\begin{equation}\label{Gstirling}
   {\mathcal G}_\epsilon(s)   \ll  (1+ |s|)^{4\Re{s}-2},
\end{equation}
uniformly for $s$ away from poles in any vertical strip of finite width.

We conclude this section by stating the Kuznetsov summation formula \cite{Ku} in the two different versions  we will apply it.  Both involve the integral kernels
\begin{equation}\label{ker}
\aligned
\mathcal{J}^{+}_t(x)  &  :=  \frac{\pi i}{ \sinh(\pi t)} (J_{2it}(x) - J_{-2it}(x)), \\
\mathcal{J}^{-}_t(x)  &  :=  4 \cosh(\pi t) K_{2it}(x)   =   \frac{\pi i}{ \sinh(\pi t)} (I_{2it}(x) - I_{-2it}(x)),
\endaligned
\end{equation}
as well as the  ``holomorphic'' kernel
\begin{equation}\label{ker1}
\mathcal{J}^{\text{hol}}_k(x)   :=   \mathcal{J}^+_{(k-1)/(2i)}(x)   =  2\pi i^k J_{k-1}(x), \quad k \in 2\Bbb{N}.
\end{equation}
Let $h(t)=O((1+|t|)^{-3})$ be an even function which is holomorphic in $|\Im t| \leq 1/2$, and let $n, m \geq 1$.  Then (see, e.g., \cite[Lemma 3.3]{BK})
\begin{multline}\label{kuz1}
\sum_{\pi} \epsilon_{\pi} \frac{\lambda_{\pi}(n) \lambda_{\pi}(m)}{L(1, \text{Sym}^2 \pi)}h(t_{\pi})   +   \frac{1}{2\pi}  \int_{-\infty}^{\infty} \frac{\lambda_{t}(n) \lambda_{t}(m)}{|\zeta(1 + 2 it)|^2} h(t) dt\\
 =     \sum_{c}\frac{S(- n, m, c) }{c} \int_{-\infty}^{\infty} \mathcal{J}^{-}_{t}\left(\frac{4\pi \sqrt{nm}}{c}\right) h(t)   d_{\text{spec}}t.
\end{multline}
Here $S(n, m, c)$ is the usual Kloosterman sum, which satisfies the Weil bound $|S(n, m, c)|    \leq  c^{1/2} (n,  m, c)^{1/2} \tau(c)$, where $\tau(c)$ is the number of positive divisors of $c$.

  Finally, formula (\ref{kuz1}) can be inverted as follows.  Suppose that $\phi \in C^3((0, \infty))$ satisfies $x^j \phi^{(j)}(x) \ll \min(x, x^{-3/2})$ for $0 \leq j \leq 3$, and let $n, m \in \Bbb{N}$. Then \cite[Theorems 2.3 and 2.5]{Mo} states that
\begin{equation}\label{kuz2}
\begin{split}
\sum_{c} \frac{S(\pm n, m, c)}{c}\phi&\left(\frac{4\pi   \sqrt{nm}}{c}\right)    =
\\ &   \sum_{\pi}  \frac{\lambda_{\pi}(\pm n) \lambda_{\pi}(m)}{L(1, \text{Sym}^2 \pi)} \int_0^{\infty} \mathcal{J}_{t_{\pi}}^{\pm}(x) \phi(x) \frac{dx}{x} \\
&  +  \frac{1}{2\pi} \int_{-\infty}^{\infty} \frac{\lambda_{t}(n) \lambda_{t}(m)}{|\zeta(1 + 2 it)|^2}\int_0^{\infty} \mathcal{J}_t^{\pm}(x) \phi(x) \frac{dx}{x}  dt\\
&  +   \left.\sum_{\pi}\right.^{\text{hol}} \frac{\lambda_{\pi}(\pm n) \lambda_{\pi}(m)}{L(1, \text{Sym}^2 \pi)}  \int_0^{\infty} \mathcal{J}^{\text{hol}}_{k_{\pi}} (x) \phi(x) \frac{dx}{x},
\end{split}
\end{equation}
where the first $\pi$-sum runs over  automorphic representations associated to cuspidal Maa{\ss} forms  for  ${\rm SL}(2,\Bbb{Z})$ having spectral parameter $t_{\pi}$, and the last $\pi$-sum runs over  automorphic representations associated to classical holomorphic cusp forms  for  ${\rm SL}(2,\Bbb{Z})$ having weight $k_{\pi} \in 2\Bbb{N}$ (with the convention $\lambda_{\pi}(-n) = 0$, so that it disappears in the minus sign case).

 \section{Balanced Voronoi summation and proof of Theorem~\ref{thm4}}\label{sec3}

 In this section we prove Theorem~\ref{thm4} using the key relations \cite[Prop.~3.6]{MS} between  tempered distributions $\sigma_{j,N,(k_1,k_2,k_3)}$ and  $\rho_{j,N,(k_1,k_2,k_3)}$ on $\Bbb{RP}^1$ defined in \cite[(2.47)]{MS}, where $1\leq j \leq 3$, $N>0$, and all subscripts are integers.
  Theorem~\ref{thm4} will be shown as a consequence of the following ``balanced'' Voronoi formula, which itself follows from the ${\rm GL}(4,\Bbb{Z})\backslash {\rm GL}(4,\Bbb{R})$ Voronoi formula in \cite{MS}.

 \begin{theorem}\label{thm:balanced}
 For any cuspidal automorphic form $\Pi$ on the quotient ${\rm GL}(4,\Bbb{Z})\backslash {\rm GL}(4,\Bbb{R})$, define
 \begin{equation}\label{curlyLdef}
   {\mathcal L}_{N,\epsilon}(s,\Pi)  :=   \sum_{m\mid N}\sum_{n\neq 0}a_\Pi(n,m,1)S\Big(n,1,\frac Nm\Big)m^{1-2s}\operatorname{sgn}(n)^\epsilon|n|^{-s},
 \end{equation}
 where $N$ is a positive integer,  $\epsilon\in\Bbb{Z}/2\Bbb{Z}$, and $\Re{s}>1$ (where the sum converges absolutely because of (\ref{RSbound})).   Then ${\mathcal L}_{N,\epsilon}(s,\Pi)$ has an analytic continuation to an entire function in $s$ and satisfies the functional equation
 \begin{equation}\label{curlyLFE}
   {\mathcal L}_{N,\epsilon}(s,\Pi) =   N^{2-4s}{\mathcal G}_\epsilon(1-s) {\mathcal L}_{N,\epsilon}(1-s,\widetilde\Pi),
 \end{equation}
 where ${\mathcal G}_\epsilon(\cdot)$ is defined in (\ref{curlyGeps}).
 \end{theorem}

 \noindent
 As we mentioned above, this theorem has been generalized to GL($n$) in \cite{Zh,MZ}, where (\ref{curlyLFE}) is derived using Dirichlet series methods.  We will include a different argument here for $n=4$, based on the machinery of \cite{MS} used to prove the usual GL(4) Voronoi formula.  Roughly speaking, the argument of \cite{MS} uses a chain of distributions on $\Bbb{R}$ -- one for each simple root of GL($n$) -- that are related using simple Weyl reflections.  The  GL($n$) Voronoi formula there is an identity between the distributions for the first and last simple roots.  In contrast, the argument here starts in the middle of the chain of distributions, and then separately links them to the  distributions for the first and last roots.

  Before giving its proof, we briefly see how Theorem~\ref{thm:balanced} implies Theorem~\ref{thm4}.
The sum (\ref{curlyLdef}) satisfies the bounds
 \begin{equation*}
    {\mathcal L}_{N,\epsilon}(s,\Pi)   \ll_\varepsilon
    \left\{
      \begin{array}{ll}
        N^{1/2+\varepsilon}, & \Re s  > 1,\\
        N^{1/2+2-2\Re{s}+\varepsilon}, & 0\leq \Re s \leq 1, \\
         N^{1/2+2-4\Re{s}+\varepsilon}, & \Re s<0,
      \end{array}
    \right.
 \end{equation*}
where the first estimate follows from (\ref{RSbound}) and Weil's bound for Kloosterman sums, the third from the functional equation (\ref{curlyLFE}), and the second from convexity (it is easy to see ${\mathcal L}_{N,\epsilon}(s,\Pi)$ has finite order).
It follows that $${\mathcal V}_\epsilon(s,z)  =  \sum_{N>0}N^{2s-1}  {\mathcal L}_{N,\epsilon}(s+z,\Pi)$$ is entire in $\{\Re(s+z)>0,\Re{z}>5/4,\Re{s}<-1/4\}\cup \{\Re(s+z)\leq 0,\Re(s+2z)>5/4\}=\{\Re(s+2z)>5/4,\Re{z}>5/4,\Re{s}<-1/4\}$, and inherits the functional equation (\ref{voronoiFE}) from (\ref{curlyLFE}).

 \begin{proof}[Proof of Theorem~\ref{thm:balanced}]
 In what follows we mainly follow the notation conventions of \cite{MS}. The third statement in \cite[Prop.~3.6]{MS}  reads
 \begin{equation}\label{voronoiproof1}
  \sigma_{2,N,(1,0,1)}(x)   =    \chi_2(Nx)  \rho_{1,N,(1,0,1)}\Big(\frac{1}{N^2 x}\Big),
\end{equation}
 where   $\chi_j: \Bbb{R}^* \longrightarrow  \Bbb{C}^*$ denotes the homomorphism
\begin{equation*}
\label{mudef} \chi_j(x) =
|x|^{\mu_j-\mu_{j+1}-1}\operatorname{sgn}(x)^{\beta_j+\beta_{j+1}}.
\end{equation*}
By \cite[Prop.~2.51]{MS} we have
\begin{equation}\label{voronoiproof2}
\aligned
    \rho_{1,N,(1,0,1)}(y) & =   \sum_{\ell \, (\text{mod }  N)}e\Big(\frac{\ell}{N}\Big)({\mathcal F}\sigma_{1,N,(0,\ell,1)})(Ny)  \\
     \text{and} \qquad
    \sigma_{2,N,(1,0,1)}(r)   & =  \sum_{\ell \, (\text{mod } N)}e\Big(-\frac{\ell}{N}\Big)({\mathcal F}\rho_{2,N,(1,\ell,0)})(-Nr)
    ,
\endaligned
\end{equation}
where $(\mathcal F\tau)(x)=\int_{\Bbb{R}}\tau(r)e(-xr)dr$ denotes the Fourier transform of a tempered distribution $\tau$.

Define the normalized coefficients
\begin{equation}\label{normalizedcoeffs}
\begin{split}
  c_{r,d,1} & =  a_\Pi(r,d,1)d^{-\mu_1-\mu_2}\operatorname{sgn}(r)^{\beta_1}|r|^{-\mu_1} \quad \text{and} \\
  & \quad c_{1,d,r} =  a_\Pi(1,d,r)d^{-\mu_1-\mu_2}\operatorname{sgn}(r)^{\beta_4}|r|^{\mu_4}
\end{split}
\end{equation}
for integers $d>0$ and $r\neq 0$ (see \cite[(2.9)]{MS}), and the distributions
\begin{equation}\label{voronoiproof46}
\aligned
    \Delta_{L;(N,\ell),1,\frac{\bar{\ell}}{N/(N,\ell)}}  & =   \sum_{n\neq 0}c_{n,(N,\ell),1}e\Big(\frac{n\bar{\ell}}{N/(N,\ell)}\Big)\delta_n \\
    \text{and} \ \ \ \
        \Delta_{R;(N,\ell),1,\frac{\bar{\ell}}{N/(N,\ell)}}   & =   \sum_{n\neq 0}c_{1,(N,\ell),n} e\Big(-\frac{n\bar{\ell}}{N/(N,\ell)}\Big) \delta_n,
\endaligned
\end{equation}
where $\delta_r\in C^{-\infty}(\Bbb{R})$ denotes the Dirac $\delta$-function supported at $r\in \Bbb{R}$.
In terms of the  operators
\begin{equation*}
\aligned
   \Big({\mathcal T}^*_{j,a,b}\sigma\Big)(x)   &  =   \operatorname{sgn}(a x)^{\beta_j+\beta_{j+1}}  | a x|^{\mu_j-\mu_{j+1}-1} (\mathcal F\sigma)\big(\frac{b}{ ax}\big)
     \\
     &=  \chi_j(ax)(\mathcal F\sigma)\big(\frac{b}{ax}\big)
\endaligned
\end{equation*}
defined for $j=1,2,3$,
we thus have
 \begin{equation}\label{voronoiproof3}
 \aligned
    \sigma_{2,N,(1,0,1)}   & =   \sum_{\ell \, (\text{mod } N)} e\Big(\frac{\ell}{N}\Big){\mathcal T}^*_{2,N,1}\sigma_{1,N,(0,\ell,1)}\\
    & =   \sum_{\ell \, (\text{mod } N)} e\Big(\frac{\ell}{N}\Big){\mathcal T}^*_{2,N,1}{\mathcal T}^*_{1,\frac{N}{(N,\ell)},\frac{(N,\ell)}{N}}\Delta_{L;(N,\ell),1,\frac{\bar{\ell}}{N/(N,\ell)}},
    \endaligned
 \end{equation}
where in the first step we have used (\ref{voronoiproof1}) and the first equation of (\ref{voronoiproof2}), and in the
second step we have invoked the first definition in (\ref{voronoiproof46}) and \cite[(4.10)]{MS}.
We also have that
\begin{equation}\label{voronoiproof5}
    \rho_{2,N,(1,\ell,0)}   =  {\mathcal T}^*_{3,\frac{N}{(N,\ell)},-\frac{(N,\ell)}{N}}\Delta_{R;(N,\ell),1,\frac{\bar{\ell}}{N/(N,\ell)}},
\end{equation}
as follows from the second formula in \cite[Prop. 3.6]{MS}.

Identities (\ref{voronoiproof2}), (\ref{voronoiproof46}),   (\ref{voronoiproof3}), and (\ref{voronoiproof5}) are all equalities of tempered distributions which   {\em vanish to infinite order} at $x=0$ and $x=\infty$ in the sense described in \cite{inforder} and \cite[Prop.~3.6]{MS}.
A distribution $\tau$ on $\Bbb{R}$ vanishing to infinite order both at zero and at infinity has an entire signed Mellin transform
\begin{equation*}
  (M_\epsilon\tau)(s)   =   \int_{\Bbb{R}}\tau(x)|x|^{s-1}\operatorname{sgn}(x)dx
\end{equation*}
(see \cite[Theorem~4.8]{inforder}).
 Thus for precisely the same analytic reasons as in \cite[(1.11)]{MS},  $(M_\epsilon\sigma_{2,N,(1,0,1)})(s)$ is entire.  (In fact, the distributional identities here are finite linear combinations of the ones there, so no additional analytic overhead is needed beyond what is provided there.)  Likewise, the signed Mellin transforms
\begin{equation*}
  (M_\epsilon  \Delta_{L;(N,\ell),1,\frac{\bar{\ell}}{N/(N,\ell)}})(s)  =
\sum_{n\neq 0}c_{n,(N,\ell),1}e\Big(\frac{n\bar{\ell}}{N/(N,\ell)}\Big) \operatorname{sgn}(n)^\epsilon
|n|^{s-1}
\end{equation*}
and
\begin{equation*}
  (M_\epsilon   \Delta_{R;(N,\ell),1,\frac{\bar{\ell}}{N/(N,\ell)}})(s) =  \sum_{n \neq 0}c_{1,(N,\ell),n} e\Big(-\frac{n\bar{\ell}}{N/(N,\ell)}\Big)\operatorname{sgn}(n)^\epsilon
|n|^{s-1},
\end{equation*}
both initially convergent for $\Re{s}$ sufficiently negative, analytically continue to entire functions of $s\in \Bbb{C}$.
For later reference, we compute that
\begin{equation*}
\aligned
  \sum_{\ell \, (\text{mod } N)} & e\Big(\frac{\ell}{N}\Big) (N,\ell)^{2s-1+\mu_1+\mu_2+2\mu_3}(M_{\epsilon+\beta_1+\beta_3}  \Delta_{L;(N,\ell),1,\frac{\bar{\ell}}{N/(N,\ell)}})(s+\mu_1-\mu_3) \\ & = \ \
\sum_{m|N}\sum_{n \neq 0}
  a_\Pi(n,m,1)S(n,1,N/m)m^{2s-1-2\mu_3}\operatorname{sgn}(n)^{\epsilon+\beta_3}|n|^{s-1-\mu_3} \\ &  = \ \
{\mathcal L}_{N,\epsilon+\beta_3}(1-s+\mu_3,\Pi)
\endaligned
\end{equation*}
and
\begin{equation}\label{MTRsum}
\begin{split}
  \sum_{\ell\, (\text{mod } N)} &\frac{e\Big(-\frac{\ell}{N}\Big)}{ (N,\ell)^{2s-1-\mu_1-\mu_2+2\mu_3}}(M_{\epsilon+\beta_3+\beta_4}  \Delta_{R;(N,\ell),1,\frac{\bar{\ell}}{N/(N,\ell)}})(1-s+\mu_3-\mu_4) \\
&=
{\mathcal L}_{N,\epsilon+\beta_3}(s-\mu_3,\widetilde\Pi),
\end{split}
\end{equation}
by inserting (\ref{normalizedcoeffs}) and grouping together terms with a common value $m$ of $(N,\ell)$.  In particular, this establishes the analytic continuation asserted in the Theorem.

When both a tempered distribution $\tau$ and its Fourier transform ${\mathcal F}{\tau}$ vanish to infinite order at both 0 at infinity,
\begin{equation}\label{MellvsTstar}
\begin{split}
  (M_\epsilon{\mathcal T}^*_{j,a,b} & \tau)(s)  = \\ &    \frac{ \operatorname{sgn}(-b)^{\epsilon+\beta_j+\beta_{j+1}}}{\operatorname{sgn}(a)^\epsilon}
  \frac{|b|^{s+\mu_j-\mu_{j+1}-1}}{|a|^{s}}\times \\
  &  G_{\epsilon+\beta_j+\beta_{j+1}}(1-s-\mu_j+\mu_{j+1})(M_{\epsilon +\beta_j+\beta_{j+1}}\tau)(s+\mu_j-\mu_{j+1}),
 \end{split}
\end{equation}
 where both Mellin transforms are entire and we have used the relationship
  \begin{equation}\label{MFvsM}
    (M_\epsilon \mathcal F \tau)(s)   =  (-1)^s  G_\epsilon(s)  (M_\epsilon\tau)(1-s)
  \end{equation}
  between the signed  Mellin transforms of a function and its Fourier transform
  (see \cite[Theorem~4.12 and Lemma~6.19]{inforder}).

 We now compute the entire function $(M_\epsilon\sigma_{2,N,(1,0,1)})(s)$ in two different ways.  Inserting (\ref{voronoiproof3}), applying (\ref{MellvsTstar}) twice, and moving the $\ell$-sum to the inside we obtain that $(M_\epsilon\sigma_{2,N,(1,0,1)})(s)$ equals
 \begin{equation}\label{signedMT1}
 \aligned
 &  \sum_{\ell \, (\text{mod } N)} e\Big(\frac{\ell}{N}\Big)
   \Big(M_\epsilon
  {\mathcal T}^*_{2,N,1}{\mathcal T}^*_{1,\frac{N}{(N,\ell)},\frac{(N,\ell)}{N}}\Delta_{L;(N,\ell),1,\frac{\bar{\ell}}{N/(N,\ell)}}
   \Big)(s) \\
 &   =
(-1)^{\epsilon+\beta_2+\beta_3}N^{-s} G_{\epsilon+\beta_2+\beta_3}(1-s-\mu_2+\mu_3)\\ & \qquad \times
(-1)^{\epsilon+\beta_1+\beta_3} G_{\epsilon+\beta_1+\beta_3}(1-s-\mu_1+\mu_3 ) \sum_{\ell\, (\text{mod }N)} e\Big(\frac{\ell}{N}\Big) \\ & \qquad \times  \Big(\frac{N}{(N,\ell)}\Big)^{1-2s-\mu_1-\mu_2+2\mu_3}
\Big( M_{\epsilon+\beta_1+\beta_3}\Delta_{L;(N,\ell),1,\frac{\bar{\ell}}{N/(N,\ell)}}  \Big)(s+\mu_1-\mu_3) \\
& =  (-1)^{\beta_1+\beta_2}N^{1-3s-\mu_1-\mu_2+2\mu_3} G_{\epsilon+\beta_2+\beta_3}(1-s-\mu_2+\mu_3)
\\
 & \qquad \times  G_{\epsilon+\beta_1+\beta_3}(1-s-\mu_1+\mu_3 ){\mathcal L}_{N,\epsilon+\beta_3}(1-s+\mu_3,\Pi).
 \endaligned
 \end{equation}
 On the other hand, applying  (\ref{MFvsM}) to the second identity in (\ref{voronoiproof2})   yields that $(M_\epsilon\sigma_{2,N,(1,0,1)})(s)$ equals
 \begin{equation}\label{signedMT3}
 \aligned
   &   G_\epsilon(s) N^{-s}\sum_{\ell\, (\text{mod }N)}\!\!e\Big(\!\!-\!\frac{\ell}{N}\Big)(M_\epsilon\rho_{2,N,(1,\ell,0)})(1-s) \\ &
   \qquad   =
      G_\epsilon(s) N^{-s}\sum_{\ell \, (\text{mod } N)}\!\!e\Big(\!\!-\!\frac{\ell}{N}\Big)\Big(M_\epsilon
      {\mathcal T}^*_{3,\frac{N}{(N,\ell)},-\frac{(N,\ell)}{N}}  \Delta_{R;(N,\ell),1,\frac{\bar{\ell}}{N/(N,\ell)}}
      \Big)(1-s)\\
   &\qquad   =
         G_\epsilon(s) N^{s-1-\mu_3+\mu_4}
        G_{\epsilon+\beta_3+\beta_4}(s-\mu_3+\mu_4)
      {\mathcal L}_{N,\epsilon+\beta_3}(s-\mu_3,\widetilde\Pi)
\endaligned
 \end{equation}
 using (\ref{voronoiproof5}) and (\ref{MTRsum}).
The functional equation (\ref{curlyLFE}) now follows from comparing (\ref{signedMT1}) to (\ref{signedMT3}).
 \end{proof}

 \section{Spectral reciprocity and proof of Theorem~\ref{thm3}}\label{sec4}

 In this section we prove Theorem \ref{thm3}.  For  $s$ with $\Re s \geq 0$ and $h$ as in Theorem \ref{thm3},   define
 \begin{equation} \label{Mshdef}
\begin{split}
\mathcal{M}^{\pm}(s; h)    :=   &\sum_{\pi} \epsilon_{\pi}^{(1\mp 1)/2} \frac{L(\textstyle{\frac 12}+s,   \Pi \times \pi ) }{L(1, \text{{\rm Sym}}^2 \pi)}  h(t_{\pi})\\
 & +   \frac{1}{2\pi} \int_{-\infty}^{\infty}  \frac{L(\textstyle{\frac 12} +s + it, \Pi)L(\textstyle{\frac 12}+s - it, \Pi)}{|\zeta(1 + 2 it)|^2}  h(t)  dt
\end{split}
\end{equation}
and
\begin{equation}\label{Mshdefhol}
   \mathcal{M}^{\text{hol}}(s; h)    :=    \left.\sum_{\pi}\right.^{\text{hol}}   \frac{L(\textstyle{\frac 12}+s,   \Pi \times \pi ) }{L(1, \text{{\rm Sym}}^2 \pi)} h(k_{\pi}).
\end{equation}
As before, we write $\widetilde{\mathcal{M}}^{\diamondsuit}(s; h)$ for $\diamondsuit \in \{+, -, \text{hol}\}$ for the same expressions with $\Pi$ replaced by $\widetilde{\Pi}$.  All of these sums are absolutely convergent and  holomorphic in  a half  plane  containing $\{s\mid \Re s \geq 0\}$; their values at $s=0$ specialize to the sums $\mathcal{M}^\diamondsuit(h)$ defined in (\ref{defM}) and (\ref{defMhol}). If $\Re s > 1/2$,  the $L$-functions in question can be replaced by their absolutely convergent Dirichlet series  \eqref{diri1} and \eqref{diri2}.  With further manipulations in mind, let us temporarily assume  $3/4<\Re s < 1$.

By design, applying the Kuznetsov formula (\ref{kuz1}) yields
\begin{equation}\label{afterkuz1}
\mathcal{M}^{-}(s; h)    =   \sum_{n, m, c > 0} \frac{a_{\Pi}(n, m, 1)}{(nm^2)^{1/2 + s}c}   S(- n, 1, c) H\left(\frac{4\pi \sqrt{n}}{c}\right),
\end{equation}
where
\begin{equation}\label{H}
H(x)   :=   \int_{-\infty}^{\infty} \mathcal{J}^{-}_t( x)  h(t)  d_\text{spec}t  =    \frac{2}{\pi^2}\int_{-\infty}^{\infty} K_{2it}(x) \sinh(\pi t) h(t) t \, dt.
\end{equation}
\begin{lemma}\label{lem1}
Let $C_2\in \Bbb{N}$.  There exists $C_1  > 0$ (depending on $C_2$) such that
\begin{equation}\label{boundH}
x^j \frac{d^j}{dx^j} H(x)    \ll_{C_2,h}   \min(x^{-C_2}, x^{C_2})
\end{equation}
for $0 \leq j \leq C_2$, for any $h:\{t\in \Bbb{C}: |\Im t|<C_1\}\rightarrow \Bbb{C}$  satisfying \eqref{c1}. In particular the Mellin transform $\widehat{H}(u)=\int_0^\infty H(x)x^{s-1}dx$ is holomorphic in $-C_2 < \Re u < C_2$ and is bounded by $\ll (1 + |u|)^{-C_2}$ in this region. Specifically, one can take $C_1 = 2C_2+2$.
 \end{lemma}

\begin{proof} We shall take  $C_1 =2 C_2+2$. First let $x \geq 1$.  For $0\leq j\leq C_2$
it follows from \eqref{diff} and \eqref{hm1} that
\begin{displaymath}
\begin{split}
  \frac{d^j}{dx^j} H(x)   & \ll_j    \int_{0}^{\infty}e^{\min(0, \pi t -x)} (1 + t/x)^{j + 1/10} (1 + t)^{-C_1} t  dt  \\ & \ll_j   x^{2-C_1}  \leq  x^{-j-C_2}
\end{split}
\end{displaymath}
(as can be seen by separately considering the integrals over $0\leq t \leq \frac{x}{2\pi}$ and $\frac{x}{2\pi}\leq t$).

Now let $x \leq 1$. We first apply \eqref{KI} and then
    express the $j$-fold derivative of the integral (\ref{H}) as a sum of  terms using \eqref{diff1}. In those terms containing $I_{2it + n}(x)/\cosh(\pi t)$ for $|n| \leq j$ we shift the contour down to $\Im t = -C_1 +\frac{1}{10}$, while in those terms containing $I_{-2it + n}(x)/\cosh(\pi t)$ for $|n| \leq j$ we shift the contour up to $\Im t = C_1 - \frac{1}{10}$; we then estimate each of these shifted integrals trivially using \eqref{Ibound}. Notice that the contour shifts do not cross poles, since by \eqref{c1} the zeros of $h$ cancel the poles of $ \cosh(\pi t)^{-1}$.  This gives
$$ x^j\frac{d^j}{dx^j} H(x) \ll_j x^j\int_0^{\infty} \frac{x^{2C_1-j-1/5}}{(1+t)^{2C_1 -j + 3/10}} (1+t)^{-C_1} t\, dt \ll x^{2C_1-1/5}\leq x^{C_2},$$
proving \eqref{boundH}. The assertions for $\widehat{H}(u)$ follow from (\ref{boundH}) and  integration by parts.
\end{proof}

Now let $C_2\geq 10$.  Applying \eqref{boundH} with $j=0$, we see that \eqref{afterkuz1} is absolutely convergent in $3/4<\Re s <1$. We now prepare for the second step, the balanced Voronoi formula in the form of Theorem \ref{thm4}.
 Using Mellin inversion we  first write
\begin{displaymath}
\begin{split}
\mathcal{M}^{-}(s; h)  =    \int_{\Re u=v} \widehat{H}(u)   \sum_{ n,  m, c > 0} \frac{a_\Pi(n, m, 1)}{(nm^2)^{1/2 + s}c}   S(- n, 1, c)  \left(4\pi \frac{\sqrt{n}}{c}\right)^{-u} \frac{du}{2\pi i}.
\end{split}
\end{displaymath}
This multiple sum/integral is absolutely convergent if $-C_2 < v < -1/2$ and $\Re (s+\frac v2)> 1/2$, a region which is nonempty if  $3/4<\Re s <1$. Using the functional equation of Theorem \ref{thm4} we continue our calculation as follows:
\begin{equation}\label{applyvoronoi}
\begin{split}
\mathcal{M}^{-}(s; h) &    =   \int_{\Re u=v} \widehat{H}(u)  (4\pi)^{-u}  \frac{1}{2}\left[\mathcal{V}_0\left(\frac{u}{2}, \frac{1}{2} + s\right) -\mathcal{V}_{ 1}\left(\frac{u}{2}, \frac{1}{2} + s\right) \right]   \frac{du}{2\pi i}\\
 & =  \int_{\Re u=v} \widehat{H}(u)  (4\pi)^{-u}  \frac{1}{2}\left[\mathcal{G}_0\left(\frac{1-u}{2}-s\right) \widetilde{\mathcal{V}}_{ 0}\left(-\frac{u}{2}-2s, \frac{1}{2}+s \right) \right.\\
 & \quad\quad\quad \quad\quad\quad\quad \left.-   \mathcal{G}_1\left(\frac{1-u}{2}-s\right) \widetilde{\mathcal{V}}_{  1}\left(-\frac{u}{2}-2s, \frac{1}{2}+s  \right) \right]   \frac{du}{2\pi i} \\
 & =  \frac{1}{2} \int_{\Re u=-v-4\Re{s}} \left(\mathcal{H}_{0}(u;s)   \widetilde{\mathcal{V}}_{ 0}\left(\frac{u}{2}, \frac{1}{2}+s \right)\right.   \\ & \qquad \qquad \qquad
  \qquad \qquad \qquad \quad
  \left.-   \mathcal{H}_{1}(u;s)    \widetilde{\mathcal{V}}_{  1}\left(\frac{u}{2}, \frac{1}{2}+s \right)\right) \frac{du}{2\pi i},
  \end{split}
 \end{equation}
with
\begin{equation}\label{defscriptH}
\mathcal{H} _{\eta}(u;s)   :=  \widehat{H}(-u-4s) (4\pi)^{u+4s}  \mathcal{G}_{\eta}\left(\frac{1+u}{2}+s\right),
\end{equation}
for $\eta=0$ or $1$.
By  \eqref{deltarevised2} and Lemma \ref{lem1}, this function
is holomorphic in $$\{(u, s) \in \Bbb{C}^2 \mid -C_2 < \Re(u+4s) < C_2, \, \Re(u+2s) > 0, \text{ and }
3/4<\Re{s}<1\},$$
and   by (\ref{Gstirling}) and again Lemma \ref{lem1}  it satisfies the estimate
$$\mathcal{H} _{\eta}(u;s)   \ll  \ (1+|u|)^{-C_2 + \Re (4s + 2u)}$$
in this region.

Having applied Voronoi summation in (\ref{applyvoronoi}), we now  prepare for the third and final step:~the application of the Kuznetsov formula in the reverse direction.  Denote by
\begin{equation}\label{defphi}
\phi_{\eta}(x; s)   =  \int_{\Re u=v} \mathcal{H}_{\eta}(u;s)   (4\pi)^{u} x^{-u} \frac{du}{2\pi i}
\end{equation}
the (slightly renormalized) inverse Mellin transform, where  
$$-2 \Re s   <    v    <    \frac{1}{2}C_2 - 2\Re s - \frac{1}{2},$$
in which case the integrand is holomorphic and the integral is absolutely convergent.  Likewise, differentiation under the integral sign gives an absolutely convergent, holomorphic integral for $x^j \phi_\eta^{(j)}(x;s)$, $j = 0, 1, 2,\ldots$, in the smaller range
\begin{equation}\label{vsmallerange}
  -2 \Re s  <   v  <    \frac{1}{2}(C_2-j-1)   -   2\Re s.
\end{equation}
Shifting the contour shows that $x^j\phi_\eta^{(j)}(x;s)=O(x^{-v})$ for any such $v$.  Since $3/4<\Re{s}<1$ and $C_2-j-1-4\Re{s}>2$ for $0\leq j \leq 3$, this easily implies the $x^j\phi_\eta^{(j)}(x;s)\ll \min(x,x^{-3/2})$ condition  for $\phi_\eta(\cdot;s)$ to be admissible in
the Kuznetsov formula (\ref{kuz2}). Thus for $v$ satisfying (\ref{vsmallerange}),
\begin{multline*}
\int_{\Re u=v}  \mathcal{H}_{\eta}(u;s)   \widetilde{\mathcal{V}}_{ \eta}\left(\frac{u}{2}, \frac{1}{2}+s \right) \frac{du}{2\pi i}  = \\ \sum_{\kappa=\pm 1} \sum_{n, m, c > 0} \kappa^\eta \frac{a_{\widetilde{\Pi}}(n, m, 1)}{(nm^2)^{1/2+s }c}  S(\kappa n, 1, c) 
 \phi_\eta\left(\frac{4\pi \sqrt{n}}{c}; s\right)
\end{multline*}
by \eqref{defV}.  Inserting into \eqref{applyvoronoi}, we find
\begin{displaymath}
\begin{split}
\mathcal{M}^-(s;h) =   & \frac{1}{2}\sum_{\kappa=\pm 1} \sum_{n, m, c > 0} \frac{a_{\widetilde{\Pi}}(n, m, 1)}{(nm^2)^{1/2+s }c}  S(\kappa n, 1, c) \\ & \qquad \times 
  \left(\phi_0\left(\frac{4\pi \sqrt{n}}{c}; s\right) -\kappa \phi_1\left(\frac{4\pi \sqrt{n}}{c}; s\right)\right) \\
\end{split}
\end{displaymath}
as an absolutely convergent expression.  Applying \eqref{kuz2} with the test function $\phi_0(\cdot; s)  -\kappa \phi_1(\cdot, s)$   gives
\begin{equation}\label{M-sh3}
\begin{split}
\mathcal{M}^{-}(s; h)   =    \widetilde{\mathcal{M}}^+(s; h_s^+)  +   \widetilde{\mathcal{M}}^{\text{hol}}(s; h_s^{\text{hol}})    +   \widetilde{\mathcal{M}}^-(s; h_s^-) ,
\end{split}
\end{equation}
where
\begin{multline*}
h_s^{\pm}(r)  :=   \frac{1}{2}\int_{0}^{\infty}  \mathcal{J}^{\pm}_r(x) \left(\phi_0(x; s) \mp \phi_1(x; s)\right) \frac{dx}{x} \\ \quad \text{and} \quad h_s^{\text{hol}}(k)   =   h_s^+\left(\frac{k-1}{2i}\right)
\end{multline*}
(see (\ref{Mshdef})--(\ref{Mshdefhol}) and (\ref{ker1})).

The functions $h_s^\pm(r)$ can be rewritten using
Parseval  as
\begin{equation}\label{useParsevalnew}
h_s^{\pm}(r)
 =      \frac{1}{2
 } \int_{\Re u=v} \widehat{\mathcal{J}}^{\pm}_{r}(u
  )  \left(\widehat{\phi}_{0} (-u; s )  \mp   \widehat{\phi}_{1} (-u; s )\right) \frac{du}{2\pi i}
\end{equation}
for small $v>0$, where the Mellin transform $\widehat{\phi}_j(\cdot; s)$ is taken with respect to the first variable.
  In fact,
applying (\ref{Gstirling}) and  \eqref{defscriptH}--\eqref{defphi}  gives the bounds
\begin{multline}\label{HGbound}
  \widehat{\phi}_{0} (-u; s )  \mp   \widehat{\phi}_{1} (-u; s )    \\
  \qquad\qquad =
   \widehat{H}(u-4s) (4\pi)^{-2u+4s}
 \left( \mathcal{G}_{0}\left(\frac{1-u}{2}+s\right) \mp   \mathcal{G}_{1}\left(\frac{1-u}{2}+s\right)\right)  \\
 \ll     (1+|u|)^{4\Re s - 2 \Re u - C_2},
\end{multline}
with an implied constant that depends locally uniformly on $\Re u$ and $s$ throughout the region $4\Re{s}-C_2 < \Re u < 2\Re s + 2\delta$ (in which the left hand side is holomorphic by \eqref{deltarevised2}).

The formula (\ref{formula}) would follow    as the $s=0$ case of (\ref{M-sh3}), were it not for the fact that our present analytic continuation of that formula holds only for   $\Re{s}>3/4$.  We shall now continue its terms to a right half plane containing $s=0$.
An application of
 \eqref{kerhat} and Stirling's formula shows that the integrand in (\ref{useParsevalnew}) is
  \begin{equation}\label{bound-int}
\ll   (1 + |u|)^{-\Re u -1- C_2 + 4\Re s},
  \end{equation}
with an implied constant that depends locally uniformly on $\Re u$, $s$, and $r$. Thus (\ref{useParsevalnew}) is valid for $0<v<2\Re{s}+2\delta$.
Therefore taking $v = \delta$ shows that (\ref{useParsevalnew}) provides an absolutely and locally uniformly convergent expression for $h_s^\pm(r)$ as a holomorphic function in $\{s\mid -\delta/2 < \Re{s}<1\}$.
  In particular, $h_s^\pm(r)$ is bounded for $r\leq 1$, locally uniformly in $s$ in that range.

We next show that $h_s^\pm(r)$ decays sufficiently rapidly for the convergence of the sums $\widetilde{\mathcal{M}}^\pm(s; h_s^\pm)$ 
 in (\ref{M-sh3}).
Let $n$ be an odd positive integer and $r \geq 1$. If $C_2 > 4\Re s + n$, we can shift the contour in (\ref{useParsevalnew}) to $\Re u = -n$, picking up residues at $-2m \pm 2 i r$, $m = 0, 1, \ldots, (n-1)/2$, from the poles of $\widehat{\mathcal J}_r^\pm(u)$: they contribute
\begin{equation*}
  \ll   r^{-1/2 -  C_2 + 4\Re s + 3m}, 
\end{equation*}
with an implied constant that depends on $n$ and locally uniformly on $s$.  To estimate the remaining integral over  $\Re u=-n$, we use the bound
\begin{equation}\label{pm-bound}
\widehat{\mathcal J}_r^\pm(u)   \ll    \big(( 1+ |\Im{u}+2r|)( 1 + |\Im{u}-2r|)\big)^{-(n+1)/2}
\end{equation}
 which follows from (\ref{kerhat}) and Stirling's formula; trivially applying this and (\ref{HGbound}) to the contributions from the three ranges $|\Im{u}|\leq r$, $r<|\Im{u}|\leq 4r$, and $4r<|\Im{u}|$ gives the estimate
\begin{equation}\label{r2}
\ll r^{4\Re{s}+n-C_2}  +   r^{4\Re{s}+2n-C_2-(n+1)/2}\int_r^{4r}(1+|y-2r|)^{-(n+1)/2}dy   +    r^{4 \Re s + n - C_2}.
\end{equation}
This last integral is bounded in $r$, so for any fixed $A>0$ we may  choose $C_2$ sufficiently large (keeping $n$ fixed) to arrange that $h_s^{\pm}(r)=O(|r|^{-A})$ as $r \rightarrow \infty$.   By \eqref{convex}, this ensures that the spectral sums $\widetilde{\mathcal{M}}^{\pm}(s, h_s^{\pm})$ defined in (\ref{Mshdef}) are absolutely and locally uniformly convergent   in  $-\delta/2 < \Re s <1 $, as long as $C_2$ is taken sufficiently large.

The above argument simplifies for
\begin{equation}\label{hhol}
\begin{split}
h_s^{\text{hol}}(k)  = &\frac{1}{2
 }\int_{\Re u =v} \widehat{\mathcal{J}}^{\text{hol}}_{k}(u
 )  \widehat{H}(u-4s) (4\pi)^{-2u+4s} \times \\ &\qquad  \left( \mathcal{G}_{0}\left(\frac{1-u}{2}+s\right) -   \mathcal{G}_{1}\left(\frac{1-u}{2}+s\right)\right)\frac{du}{2\pi i}.
 \end{split}
 \end{equation}
 Indeed, \eqref{bound-int} remains true for this integrand, from which we again conclude the holomorphic continuation of $h_s^{\text{hol}}(k)$ to $-\delta/2 < \Re{s}<1$.  For the decay in $k$, let $n$ be a fixed positive integer.  Since $\widehat{\mathcal{J}}^{\text{hol}}_{k}(u)=
 i^k 2^u \pi \Gamma(\frac{1}{2}(u+k-1)) \Gamma(\frac{1}{2}(1+k-u))^{-1}$ has no poles for $\Re u \geq -n$ when $k\geq n+2$, we may shift the contour to the line $\Re{u}=-n$.  Here $\Gamma(\frac{1}{2}(u+k+2n+1))=
 p_{n+1}(u+k)\Gamma(\frac{1}{2}(u+k-1))$, where $p_{n+1}(x)=2^{-n-1}(x-1)(x+1)\cdots (x+2n-1)$ is a degree $(n+1)$
 polynomial.  Hence for fixed $n$ and $\Re{u}=-n$ we have
   $$\widehat{\mathcal{J}}^{\text{hol}}_{k}(u) \ll \left|
  \frac{\Gamma(\frac{1}{2}(u+k-1))}{\Gamma(\frac{1}{2}(-u+k+1))}\right|
     \ll   (k +|\Im u|)^{-n-1},$$
  instead of the earlier bound \eqref{pm-bound}.   This shows that for any fixed  $A>0$ we may  choose $C_2$ sufficiently large   to arrange that $h_s^{\text{hol}}(k)=O(k^{-A})$ as $k \rightarrow \infty$, and in particular ensure the absolute and locally uniform convergence of $\widetilde{\mathcal M}^{\text{hol}}(s,h_s^{\text{hol}})$.


This completes the analytic continuation of (\ref{M-sh3}) to $s=0$. We obtain the spectral reciprocity formula of Theorem \ref{thm3} with $h^{\diamondsuit} = h^{\diamondsuit}_0$ using (\ref{useParsevalnew})-(\ref{HGbound}), i.e.,
\begin{equation}\label{defkernel}
\begin{split}
 \mathcal{K}^{\pm}(t, r) & =  \frac{8\pi^2}{2} \int_{\Re u = \delta } \widehat{\mathcal J}_r^{\pm}(u)\widehat{\mathcal J}^-_t(u) (4\pi)^{-2u} \times \\ &\qquad\qquad \left( \mathcal{G}_{0}\left(\frac{1-u}{2}\right) \mp   \mathcal{G}_{1}\left(\frac{1-u}{2}\right)\right) \frac{du}{2\pi i}\\
 &   =     \int_{\Re u = \delta } \widehat{\mathcal J}_r^{\pm}(u) 2^{-u} \widehat{\mathcal J}^-_t(u) 2^{-u} \mathcal{G}^{\pm}\left(\frac{1-u}{2}\right) \frac{du}{2\pi i}
 \end{split}
\end{equation}
(recall \eqref{similar}) and
$$\mathcal{K}^{\text{hol}}(t, k)    =   \mathcal{K}^{+}\left(t, \frac{k-1}{2i}\right).$$
By \eqref{kerhat} this matches the definitions \eqref{kernel1} and \eqref{kernel2}.

It remains to prove the  bound \eqref{small} for $\mathcal{K}^+$.  Recalling definition (\ref{E01}), we see that
$$\prod_{j\leq 4}\mathcal{E}_{\beta_j}(u + \mu_j) -\prod_{j\leq 4} \mathcal{E}_{\beta_j + 1}(u + \mu_j)    \ll   e^{\pi |\Im u|}$$
instead of the trivial bound $e^{2\pi  |\Im u|}$,
where the  second condition in (\ref{sumzero}) is used to match and then cancel the powers of $i$ in the leading terms on the left hand side.  Thus we have the  estimate
$$\mathcal{G}^+(u)    \ll    e^{-\pi |\Im u|}  (1+|u|)^{4\Re u - 2}$$
for  \eqref{defG}
in fixed vertical strips (see (\ref{Gstirling})). Shifting the $u$-contour in  (\ref{defkernel}) to $\Re u = -1$, the contribution of the   residues at $u = \pm 2it$ is
\begin{displaymath}
  \ll   e^{\pi |t| - \frac{\pi}{2}(2|t| + |t-r| + |t+r|)} \big((1 + |t|)(1 + |t-r|)(1 + |t+r|)\big)^{-1/2}    \leq   e^{-\pi \max(|t|, |r|)}
\end{displaymath}
and the contribution of the residues at $u = \pm 2ir$ is
\begin{displaymath}
 \ll   e^{\pi |t| - \frac{\pi}{2}(2|r| + |t-r| + |t+r|)} \big((1 + |r|)(1 + |t-r|)(1 + |t+r|)\big)^{-1/2}   \leq   \frac{e^{-\pi   |r|}}{ (1+|t|)}.
\end{displaymath}
The remaining contour integral over $u=-1+2iw$ is
\begin{displaymath}
\begin{split}
& \ll     \int_{-\infty}^{\infty}  \frac{(1 + |w|)^2 }{(1 + |w^2 - t^2|)(1 + |w^2 - r^2|)} e^{\frac{\pi}{2}(2|t| - |w + t| - |w-t| - |w - r| - |w+ r|)} dw\\
&   \qquad  \qquad \leq   \int_{-\infty}^{\infty}  \frac{(1 + |w|)^2 }{(1 + |w^2 - t^2|)(1 + |w^2 - r^2|)}  e^{-\frac{\pi}{2}(|r| + |w|)} dw   \\
&  \qquad \qquad  \qquad \qquad\ll   \frac{e^{-\frac{\pi}{2} |r|}}{ (1+ |t|^2)(1 + |r|^2)}.
\end{split}
\end{displaymath}
This completes the proof of Theorem~\ref{thm3}.\\

\textbf{Remark 1:} While $h^{-}(r)$ can be arranged to decay to any fixed polynomial order, the kernel $\mathcal{K}^-(t, r)$ itself is \emph{not} rapidly decaying in $r$; for fixed $t$ it is of order of magnitude $(1 + |r|)^{-1}$, which by Weyl's law does not suffice to make the spectral $r$-sum absolutely convergent (not even under GRH for the corresponding central $L$-values).  It is the extra integration over $t$ together with the regularity assumptions \eqref{c1} on the test function $h$ that gives adequate decay properties of $h^-$. On the other hand, the bound \eqref{small} shows that the decay of $h^{+}$ happens already on the level of $\mathcal{K}^{+}$. \\

\textbf{Remark 2:} The numerical value $C_1 \geq 40$ in Theorem \ref{thm3} arises as follows:~choosing $n = 5$ and $C_2 = 12 + 4\Re s$, we obtain from (\ref{r2}) that $h_s^{-}(r) \ll (1+|r|)^{-5}$, which by Weyl's law and \eqref{convex} makes the spectral sum absolutely convergent. For $0 \leq \Re s \leq 1$, we conclude that $C_2 = 16$ is admissible and hence $C_1 = 34$ by Lemma \ref{lem1}.

\section{Interlude: description of the kernel}\label{interlude}

The proof of Theorem~\ref{thm3} in Section~\ref{sec4} shows that the passage $h \mapsto h^{\diamondsuit}$ arises from  three steps.

\emph{Step 1 [Kuznetsov (\ref{kuz1})]}: Define
$$H(x)    :=  \int_{-\infty}^{\infty} \mathcal{J}_t^-(x) h(t)d_{\text{spec}}t.$$

\emph{Step 2 [Voronoi (\ref{voronoiFE})]}: Let $\widehat{H}(u)$ denote the Mellin transform of $H$ and define
\begin{equation}\label{step2}
\mathcal{H}^{\pm}(u)   :=    \widehat{H}(-u)  4^u\mathcal{G}^{\pm}\left(\frac{1+u}{2}\right) .
\end{equation}

\emph{Step 3 [Kuznetsov in reverse direction (\ref{kuz2})]}: Let $\phi^{\pm}(x)$ denote the inverse Mellin transform of $\mathcal H^\pm(u)$. Then
\begin{multline}\label{hminusphi}
  h^{\pm}(r)   =    \frac{1}{8\pi^2} \int_0^{\infty} \mathcal{J}^{\pm}_r(x) \phi^{\pm}(x) \frac{dx}{x}    \quad \\ \text{and}\qquad  h^{\text{hol}}(k)    =  \frac{1}{8\pi^2} \int_0^{\infty} \mathcal{J}^{\text{hol}}_k(x) \phi^+(x) \frac{dx}{x}.
\end{multline}
This follows from \eqref{H}, \eqref{useParsevalnew}, \eqref{HGbound},  \eqref{hhol} with $s=0$, and \eqref{similar}.

For completeness we give two additional, alternative expressions for $\mathcal{K}^{-}(t, r)$ (similar expressions hold for  $\mathcal{K}^{\diamondsuit}(t, r)$ in general). Let
$$g(x)  :=   \int_{\mathcal{C}}   \mathcal{G}^{-}(u) x^{-u} \frac{du}{2\pi i},$$
where  the contour $\mathcal{C}$ is taken to the right of  $\Re u = 1/2-\delta$ for $u$ small (to avoid poles -- see (\ref{deltarevised2})), and a bounded distance to the left of $\Re{u}=1/4$ for $u$ large (to converge absolutely -- see \eqref{Gstirling}).
Then starting from (\ref{defkernel}),
\begin{displaymath}
\begin{split}
  \mathcal{K}^{-}(t, r) &   =     \int_{\Re u =\delta}  \mathcal{G}^{-}\left(\frac{1-u}{2}\right)  \widehat{\mathcal{J}}^{-}_t(u)\widehat{\mathcal{J}}^{-}_r(u) 4^{-u} \frac{du}{2\pi i}\\
   &    =     2 \int_{\Re z =\frac{1-\delta}{2}}  \mathcal{G}^{-}(z) \widehat{\mathcal{J}}^{-}_t(1-2z)\widehat{\mathcal{J}}^{-}_r(1-2z) 4^{2z-1} \frac{dz}{2\pi i}\\
  & =    2 \int_{\mathcal{C}}   \int_0^{\infty} \int_0^{\infty} \mathcal{J}^{-}_t(x) \mathcal{J}^{-}_r(y)  (xy)^{-2z}  dx \, dy\, \mathcal{G}^{-}(z) 4^{2z-1} \frac{dz}{2\pi i}\\
  & = \frac{1}{2}  \int_0^{\infty} \int_0^{\infty} \mathcal{J}^{-}_t(x) \mathcal{J}^{-}_r(y) g\! \left(\frac{(xy)^2}{16}\right) dx \,dy,
\end{split}
\end{displaymath}
where all integrals are absolutely convergent (as can be seen from standard bounds on  $\mathcal J_t^-$  and $\widehat{\mathcal{J}}_t^{-}$ along the lines of  (\ref{hm1}) and (\ref{kerhat})).
In practice, one may want to asymptotically evaluate  $g(x)$  as in \cite[(4.11)]{Miller} or \cite[Lemma 6.1]{Li} (e.g., using stationary phase). 

Alternatively, one can shift the $u$-contour  far to the left and pick up the residues at $-2n \pm 2it$ and $-2n\pm 2ir$ for $n \in \Bbb{Z}_{\geq 0}$, which can be written as a sum of quotients having four $\Gamma$-factors  in the numerator and three $\Gamma$-factors in the denominator. Hence the sum of residues can be evaluated in terms of  $_4F_3$ hypergeometric functions \cite[(9.14)]{GR}.  We now describe this in the special case that the $\beta_j$ from (\ref{sumzero}) all  vanish, the other cases being similar. Let
 \begin{displaymath}
\begin{split}
K&(t, r)  :=  \\ & 8\cosh(\pi t) \cosh(\pi r)  \Gamma(2ir)   \Gamma(i(r+ t)) \Gamma(i(r- t))\prod_{j=1}^4  \Gamma\left(\frac{1}{2} + \mu_j - ir \right)   \\
& \times  \left[\cos(\pi(\mu_2 +\mu_3)) + \cos(\pi(\mu_2 +\mu_4)) + \cos(\pi(\mu_3 +\mu_4)) -\cosh(2\pi r)\right]\\
& \times  _4F_3\left(\begin{array}{cccc}\frac{1}{2} + \mu_1 - ir &\frac{1}{2} + \mu_2 - ir & \frac{1}{2} + \mu_3 - ir & \frac{1}{2} + \mu_4 - ir\\ 1 - 2 i r & 1 - i(r+t) & 1 - i(r - t) & \end{array}, 1\right)
\end{split}
\end{displaymath}
for $r, t\in \Bbb{R}$, $rt(t^2 - r^2) \not= 0$. With this notation,we have
$$\mathcal{K}^{-}(t, r)   :=     K(t, r)  +    K(t, -r)  +   K(r, t)   +  K(r,- t)$$
 if $rt(t^2 - r^2) \not= 0$ (otherwise, $\mathcal{K}^{-}(t, r)$ is defined by continuity).


\section{Asymptotic analysis and proof of  Theorem \ref{thm1} }\label{sec5}

 In this section we prove Theorem \ref{thm1}. Let   $h_T$ be as in \eqref{defH}, with $T $ very large and $D \geq 50$ fixed.  We recall the definitions \eqref{defM} and \eqref{defMhol} of $\mathcal{M}^{\diamondsuit}$ for $\diamondsuit \in \{+, -, \text{hol}\}$, as well as the reciprocity formula \eqref{formula} (where we use $\widetilde{\mathcal{M}}^{\diamondsuit}$ to denote the analogous quantities for the dual representation $\widetilde\Pi$).    The bound \eqref{small} together with \eqref{convex} shows that trivially
 $$|\widetilde{\mathcal{M}}^+(h_T^+)|     \ll    \int_{-\infty}^{\infty} h_T(t)    (1+|t|)^{-1}   t \, dt   \ll   T.$$
By \eqref{formula} and \eqref{convex} it suffices to show
\begin{lemma}\label{lem2} If $D$ in \eqref{defH} is sufficiently large, then
\begin{equation*}
  h_T^-(r)   \ll   T \min(|r|^{-1}, |r|^{-5})
\end{equation*}
for $r \in \Bbb{R}$ and
$$h_T^{\text{{\rm hol}}}(k)    \ll   Tk^{-5}$$
for integers $k \geq 2$.
  \end{lemma}

The rest of this section is devoted to the proof of Lemma \ref{lem2}.  We  follow the   three steps of Section \ref{interlude}, beginning with an analysis of (\ref{H}),
\begin{equation*}
  H_T(x)    =    \int_{-\infty}^{\infty} \mathcal{J}_t^-(x) h_T(t)  d_{\text{spec}}t    =  4\int_{-\infty}^{\infty} K_{2it}(x)\sinh(\pi t) h_T(t)  t  \frac{dt}{2\pi^2},
\end{equation*}
in the $T$-aspect. As an analogue of Lemma \ref{lem1} we obtain the following result  whose proof uses similar ideas as \cite[Lemma 3.8]{BK}.

\begin{lemma}\label{lem3} Let $j \in \Bbb{Z}_{\geq 0}$ and suppose that $D \geq \max(7, j)$. Then
\begin{equation}\label{lemma4equation}
  x^j \frac{d^j}{dx^j} H_T(x)   \ll_{D, j}   T \min\left(\left(\frac{x}{T}\right)^{D/2}, \left(\frac{x}{T}\right)^{-D/2}\right).
\end{equation}
\end{lemma}
%

\begin{proof} 
Fix $D$, $j$, and three positive integers $A_1, A_2, A_3$ (which will be chosen later); all implied constants in the proof may depend on these parameters.  We argue separately in three ranges for $x$.\\

\noindent
\emph{Range I:  $x \leq 1$.} We proceed similarly as in the proof of Lemma \ref{lem1} by applying \eqref{KI}   and differentiating under the integral sign   using \eqref{diff1}. 
In those terms containing $I_{2it + n}(x)/\cosh(\pi t)$ for $|n| \leq j$ we shift the contour  to $\Im t = -D$, while in those terms containing $I_{-2it + n}(x)/\cosh(\pi t)$ for $|n| \leq j$ we shift the contour   to $\Im t = D$. The polynomial $P_T$ in the definition \eqref{defH} ensures that no poles are crossed during the contour shifts. Now we  estimate trivially using \eqref{Ibound} 
and obtain 
$$\aligned x^j \frac{d^j}{dx^j} H_T(x)  & \ll  x^j \int_0^{\infty} e^{-t^2/T^2} \left(\frac{1+t}{T}\right)^{4D} \frac{x^{2D-j}}{(1+t)^{2D-j+1/2}} t  dt \\ &\ll  \frac{x^{2D}}{T^{2D-j- 3/2}}   \leq T\left(\frac{x}{T}\right)^{D/2}
\endaligned$$
for $x \leq 1$ and $D \geq \max(7, j) $. \\

\noindent
\emph{Range II: $1 \leq x \leq T^{13/12}$.} Let
$$h_{\text{spec}}(t)   :=   h_T(t)\frac{d_{\text{spec}}t}{dt}  =  \frac{1}{2\pi^2}  h_T(t) t \tanh(\pi t),$$
which has  Fourier transform
\begin{equation}\label{fourier-hspec}
\check{h}_{\text{spec}}(v)   =  \int_{-\infty}^{\infty} h_{\text{spec}}(t) e^{-itv} dt   =  \frac{T^2}{2\pi^2} \int_{-\infty}^{\infty} e^{-t^2}  P_T(tT) \tanh(\pi t T) e^{-itTv} t \, dt.
\end{equation}
We note that
\begin{displaymath}
\begin{split}
\frac{d^n}{dt^n} P_T(tT) &   \ll_n   T^{-4D + n} (1 +  |t|T)^{4D-n} \quad (n \geq 0), \\
 \frac{d^n}{dt^n}  \tanh(\pi t T) &   \ll_n   \frac{T^n}{\cosh(\pi t T)^2} \quad (n \geq 1),
 \end{split}
 \end{displaymath}
and clearly $ \frac{d^n}{dt^n}e^{-t^2}t^{j+1} \ll_{n,j} e^{-t^2} (1 + |t|)^{n+j+1}$, so that by the Leibniz rule
$$\frac{d^n}{dt^n} e^{-t^2} P_T(tT)\tanh(\pi t T) t   \ll_n    e^{-t^2} (1 + |t|)^{n+1}T^{-4D + n} (1 +  |t|T)^{4D-n}$$
for all integers $n,j\geq 0$. Applying $\frac{d^j}{dv^j}$  and then integrating by parts $A_1$ times,
  we have
\begin{equation*}
\aligned
\check{h}^{(j)}_{\text{spec}}(v) & \ll  \frac{T^{2+j} }{(T|v|)^{A_1}} \int_{-\infty}^{\infty}  e^{-t^2} (1+|t|)^{A_1 + j+1}  T^{-4D + A_1} (1 +  |t|T)^{4D-A_1}  dt\\ & \ll   \frac{T^{2+j} }{ (T|v|)^{A_1}}
\endaligned
\end{equation*}
for any fixed $  j \geq 0$ and any $A_1 \leq 4D$.  Combining this with the case $A_1=0$, we deduce
\begin{equation}\label{combine}
\check{h}^{(j)}_{\text{spec}}(v)  \ll     T^{2+j} (1 + T|v|)^{-A_1}
\end{equation}
for $j \geq 0$ and $A_1 \leq 4D$. Alternatively, we apply $\frac{d^j}{dv^j}$ to  (\ref{fourier-hspec})    and then shift the contour to $\Im t =D$ (if $v < 0$) or $\Im t = -D$ (if $v > 0$); this crosses no poles because $h_T$ vanishes at the poles of $\tanh(\pi t)$ in this region.
Therefore
\begin{equation}\label{combine1}
 \check{h}^{(j)}_{\text{spec}}(v)  \ll     T^{2+j} e^{-D|v|}
 \end{equation}
(as is also clear from (\ref{defH}) and Paley-Wiener theory).

Having the bounds \eqref{combine} and \eqref{combine1} available,
we shall now derive an alternative expression for $H_T(x)$ and its derivatives (in \eqref{v-int2} below) that will be easier to analyze since it features the highly-localized Fourier transform $\check{h}_{\text{spec}}$ instead of $h_T$ itself.  
Applying \eqref{four} shows
\begin{equation}\label{v-int1}
H_T(x)  =    \int_{-\infty}^{\infty}{\mathcal J}_t^{-}(x) h_{\text{spec}}(t) dt  =     \int_{-\infty}^{\infty}e^{ix \sinh (v/2)}  \check{h}_{\text{spec}}(v)  dv.
\end{equation}
Since
$$x \frac{d}{dx} e^{i x \sinh (v/2)}   =   2\tanh(v/2)\frac{d}{dv} e^{i x \sinh (v/2)}, $$
we can combine each differentiation in   $x$ with an integration by parts  to obtain
\begin{equation*}
\left(x \frac{d}{dx}\right)^j H_T(x)  =   \int_{-\infty}^{\infty}e^{ix \sinh (v/2)}   \mathcal{D}^j\check{h}_{\text{spec}}(v) dv,
\end{equation*}
where
$$\mathcal{D}f(v)  :=  - 2\frac{d}{dv} (\tanh (v/2) f(v))$$
(these integrals all converge absolutely by  \eqref{combine1}).
The $j$-fold composition $\mathcal D^j$ can be expressed as a finite sum
 $$\mathcal D^j  =  \sum_{n=0}^j\sum_{a\geq 0}\sum_{b\geq n}\tilde{\rho}_{a,b,n}\operatorname{sech}(v/2)^a \tanh(v/2)^b \frac{d^n}{dv^n},$$
  with $\tilde{\rho}_{a, b, n} \in \Bbb{R}$.
  Since the differential operators   $\frac{d}{dx}$ and $x$ have commutator $\frac{d}{dx}x-x\frac{d}{dx}=1$, by moving all occurrences of $\frac{d}{dx}$ as far to the right as possible  we obtain the expression
  \begin{equation}\label{v-int2}
x^j \frac{d^j}{dx^j} H_T(x)    =     \int_{-\infty}^{\infty}e^{ix \sinh (v/2)}   (p_j(\mathcal{D})\check{h}_{\text{spec}})(v) dv,
\end{equation}
  where $p_j$ is a polynomial of degree $j$ and
  \begin{equation}\label{pjDsum}
  p_j(\mathcal D)   =  \sum_{n=0}^j\sum_{a\geq 0}\sum_{b\geq n} \rho_{a,b,n}\operatorname{sech}(v/2)^a \tanh(v/2)^b \frac{d^n}{dv^n}  ,  \quad   \rho_{a, b, n}  \in  \Bbb{R},
  \end{equation}
is a finite sum.
  Using the boundedness of $\operatorname{sech}(v/2)$ we thus have
 that
 \begin{equation}\label{Djhcheckbd}
  (p_j(\mathcal D) \check{h}_{\text{spec}})(v)    \ll
   \sum_{c = 0}^j |\check{h}^{(c)}_{\text{spec}}(v) \tanh(v/2)^{c}|
   \ll   T^{2} (1 + T|v|)^{-A_1+j}
 \end{equation}
for $j+2 \leq A_1 \leq 4D$  by \eqref{combine},  where in the first inequality we used the boundedness of $\tanh(v/2)$ and in the second inequality we instead used the bound $\tanh(v/2)\ll v$.

 In light of the estimate (\ref{Djhcheckbd}), we can truncate the $v$-integral \eqref{v-int2} at $|v| \leq T^{-3/4}$ at the cost of an error
\begin{equation}\label{e0}
  \ll   T^2 \int_{T^{-3/4}}^{\infty} (1 + Tv)^{-A_1+j} dv    \ll    T^{  2   + (j-A_1 -3)/4}.
   \end{equation}
From now on we assume $|v| \leq T^{-3/4}$ and approximate the integrand in \eqref{v-int2} by various Taylor expansions in a neighbourhood of $v=0$.
Since $x \leq T^{13/12}$, we have $|xv^3| \leq T^{-7/6}$,
 so we can write 
$$  e^{i x \sinh (v/2) }   =   e^{ixv/2} \sum_{\beta = 0}^{A_2} \frac{(i x)^{\beta} (\sinh (v/2) - v/2)^\beta}{\beta!}  + O\left(T^{-7(A_2+1)/6}\right)$$
for any fixed positive integer $A_2$. Now writing
$$ (\sinh (v/2) - v/2)^\beta   =  \sum_{\alpha= 3\beta}^{3A_2} \tilde{c}_{\alpha, \beta} v^{\alpha} + O(T^{-3(3A_2 +1)/4})$$
for certain  constants $\tilde{c}_{\alpha, \beta} $, we obtain
\begin{equation}\label{taylor1}
 e^{i x \sinh (v/2) }   =  e^{ixv/2}  \sum_{\alpha=0}^{3A_2} \sum_{\beta= 0}^{\lfloor \alpha/3 \rfloor} c_{\alpha, \beta}  x^{\beta} v^{\alpha} + O\left(T^{-\frac{7}{6}A_2 - \frac{3}{4}}\right)
 \end{equation}
  with constants $c_{\alpha, \beta}$ (recalling that $x \leq T^{13/12}$).

   Similarly, expanding $\operatorname{sech}(v/2)^a \tanh(v/2)^b$ about $v=0$, we see from \eqref{combine} and \eqref{pjDsum} that
 \begin{multline}\label{taylor2}
 \begin{split}
 (p_j( \mathcal{D})&\check{h}_{\text{spec}})(v)  =    \\ & \sum_{n=0}^j \Bigl[ \Bigl(\sum_{\gamma = n}^{A_3} d_{n, \gamma} v^{\gamma }\Bigr)   \check{h}_{\text{spec}}^{(n)}(v) + O\left(T^{-3(A_3  + 1)/4} T^{2 + n}(1 + T|v|)^{-A_1}\right)\Bigr],
 \end{split}
  \end{multline}
       where $d_{n, \gamma}$ are constants.  Recalling (\ref{e0}) and  substituting
        (\ref{taylor1}) and (\ref{taylor2}) into the truncated version of  \eqref{v-int2}, we conclude that $x^j   H^{(j)}_T(x) $ equals
\begin{equation}\label{firststep}
\begin{split}
 \sum_{\alpha=0}^{3A_2} \sum_{\beta= 0}^{\lfloor \alpha/3\rfloor} & c_{\alpha, \beta}  x^{\beta}      \sum_{n=0}^j   \sum_{\gamma = n}^{A_3} d_{n, \gamma} \\ & \int_{-T^{-3/4}}^{T^{-3/4}} e^{ixv/2}  v^{\alpha+\gamma} \check{h}^{(n)}_{\text{spec}}(v) dv  + R_1 +  R_2   +  O(T^{2+(j-A_1-3)/4}),
 \end{split}
\end{equation}
where
\begin{equation}\label{e1}
R_1   \ll  T^{-\frac{7}{6}A_2 - \frac{3}{4}} \sum_{n=0}^j\int_{-T^{-3/4}}^{T^{-3/4}}  T^{2+n}|v|^n (1 + T|v|)^{-A_1} dv     \ll   T^{-\frac{7}{6}A_2 + \frac{1}{4}}
\end{equation}
accounts for the contribution of the error term in \eqref{taylor1} and the main term in \eqref{taylor2} (using \eqref{combine}, $|v| \leq T^{-3/4} \leq 1$, and the assumption $A_1 \geq j+2$), and
\begin{equation}\label{e2}
\aligned
 R_2  &  \ll    T^{-3(A_3 + 1)/4} T^{2 + j}  \int_{-T^{-3/4}}^{T^{-3/4}} (1 + T|v|)^{-A_1}  dv  \\ &  \ll    T^{-3(A_3 + 1)/4} T^{1 + j}    =  T^{(1-3A_3 +4j)/4}.
 \endaligned
  \end{equation}
 accounts for the  contribution of \eqref{taylor1} and the error term in \eqref{taylor2}.

 We now complete the integral in \eqref{firststep} to $(-\infty, \infty)$ at the cost of introducing an error (using \eqref{combine} and $x \leq T^{13/12}$)
\begin{equation*}
\ll   \sum_{\alpha=0}^{3A_2}   T^{13\alpha/36}     \sum_{n=0}^j   \sum_{\gamma = n}^{A_3}   \int_{T^{-3/4}}^{\infty}  v^{\alpha+\gamma} T^{2+n} (Tv)^{-A_1} dv    \ll   T^{(5-A_1+j)/4},
\end{equation*}
provided $A_1 \geq 3A_2 + A_3 + 2$.  (Note that this error coincides with the error  \eqref{e0} from truncating (\ref{v-int2}).) We now choose $A_1$ to be as large as possible,  and  $A_2, A_3$ to simultaneously nearly get  equality in this last condition    and to nearly equalize\footnote{The expressions given here have simpler fractions than the actual equalizers.} \eqref{e1} and \eqref{e2}:
$$A_1=4D, \quad  A_2 = \Big\lfloor \frac{1}{7}(6D - 2j - 3)\Big\rfloor\ , \quad \text{and} \quad A_3 = \Big\lfloor\frac{2}{21}(14D + 9j - 7)\Big\rfloor.$$
With these choices the three error terms in  \eqref{firststep} contribute
$$\aligned & \ll   T^{\frac{23}{12} + \frac{j}{3} - D} + T^{\frac{3}{2} + \frac{5}{14}j - D}+ T^{\frac{5}{4} + \frac{j}{4} - D}  \ll  T^{2 + \frac{5}{14}j - D}  \\ &\leq   T \min\left(\Big(\frac{x}{T}\Big)^{-D/2}, \Big(\frac{x}{T}\Big)^{D/2}\right)
\endaligned$$
for $1 \leq x \leq T^{13/12}$ and $\max(2, j) \leq D$.

The main term, i.e., the expression \eqref{firststep}  but with integration extended over the real line, is by Fourier inversion a linear combination of terms of the form
\begin{displaymath}
\begin{split}
x^{\beta} \frac{d^{\alpha+\gamma}}{dx^{\alpha+\gamma}} \left(x^{n} h_{\text{spec}}(x/2)\right)   \ll   \frac{x^{\beta}}{T^{4D}} e^{-x^2/T^2} x^{4D+n+1}\left(  \frac{x}{T^2} + \frac{1}{x}\right)^{\alpha+\gamma}
\end{split}
\end{displaymath}
 with $\alpha \leq 3A_2$, $\beta \leq \alpha/3$, $n \leq j \leq D$, $n \leq \gamma \leq A_3$ (where we have used that $x \geq 1$ in obtaining this bound).
 If $1 \leq x \leq T$  this is $\ll x (x/T)^{4D}$, while if $T \leq x \leq T^{13/12}$  this is
\begin{multline*}  \ll   e^{-x^2/T^2} x \left(\frac{x}{T}\right)^{4D}  \frac{x^{\beta+\alpha+\gamma+n}}{T^{2(\alpha+\gamma)}}  \ll
 e^{-x^2/T^2} x \left(\frac{x}{T}\right)^{4D}  \frac{x^{\beta+\alpha}}{T^{2\alpha}}
  \frac{x^{\gamma+n}}{T^{2\gamma}}  \\  \ll
  e^{-x^2/T^2} x \left(\frac{x}{T}\right)^{4D+2\alpha+2\gamma}
    \ll Te^{-x^2/T^2}\left(\frac{x}{T}\right)^{4D+6A_2+2A_3+1}  \\  \leq  Te^{-x^2/T^2}\left(\frac{x}{T}\right)^{12 D}\qquad\qquad\qquad\qquad\qquad\qquad\qquad,
 \end{multline*}
and
  (\ref{lemma4equation}) follows from the rapid decay of the exponential.\\

\noindent
\emph{Range III: $x \geq T^{13/12}$.} We use the rapid decay of the Bessel $K$-function:~by \eqref{diff} and  \eqref{hm1} and splitting the integral at $|t| = \frac{x}{3\pi}$, we have
\begin{equation*}
\begin{split}
x^j & \frac{d^j}{dx^j} H_T(x) \\ &  \ll  x^j \int_{-\infty}^{\infty} e^{-t^2/T^2} \left(\frac{1+ |t|}{T}\right)^{4D} e^{\min(0, -x+\pi |t|)} \left(\frac{|t| + x}{x}\right)^{j + 1/10} |t| \, dt \\
& \ll  x^j \left(\frac{x}{T}\right)^{4D} e^{-2x/3} x^2  + x^j  \int_{|t| \geq x/(3\pi)} e^{-t^2/T^2} \left(\frac{|t|}{T}\right)^{4D} \left(\frac{|t|}{x}\right)^{j+1} |t|  \, dt\\
& \ll   e^{-x/2}  \  + \  T^{j+3} x^{-1} e^{-x^2/(10T)^2}   \ll    e^{-x/2}   +   x^{j+2}e^{-x^2/(10T)^2},
\end{split}
\end{equation*}
which is certainly dominated by the desired bound $T (x/T)^{-D/2}$.
This completes the proof of  Lemma \ref{lem3}.
\end{proof}

\begin{proof}[Proof of Lemma~\ref{lem2}]
As in the proof of Lemma~\ref{lem1}, the estimates of
 Lemma \ref{lem3} imply that the Mellin transform $\widehat{H}_T(u)$ is holomorphic in $-D/2 < \Re u < D/2$; moreover taking  $j=D$ in (\ref{lemma4equation})  implies the bound
$$\widehat{H}_T(u)  \ll_{\Re u}    T^{1+\Re u} (1 + |u|)^{-D}$$
in this region (using integration by parts). Therefore the function $\mathcal{H}^{\pm}_T(u)=$ $\widehat{H}_T(-u)  4^u\mathcal{G}^{\pm}(\frac{1+u}{2})$ from \eqref{step2}  is holomorphic in $-2\delta < \Re u < D/2$ by \eqref{deltarevised2}, and satisfies the bound
$$\mathcal{H}^{\pm}_T(u)   \ll_{\Re u}    T^{1-\Re u}  (1 + |u|)^{-D + 2\Re u}$$
using (\ref{Gstirling}).
This in turn implies that the inverse Mellin transform $\phi^{\pm}_T$ of $\mathcal{H}^{\pm}_T$ satisfies
\begin{equation}\label{inverse}
\aligned
x^j  \frac{d^j}{d x^j} \phi^{\pm}_T(x) &  \ll   x^{-\Re v} \int_{\Bbb{R}} |\mathcal{H}^{\pm}_T(v+it)| (1 + |v+it|)^j  dt  \\ & \ll   T \min\left((xT)^{\delta}, (xT)^{-D/3}\right)
\endaligned
\end{equation}
upon choosing $v = -\delta$ or $v = D/3$,
provided that  $j < \frac{1}{3}D - 1$ (to ensure integrability).

It remains to bound $h^{-}_T(r)$ and $h^{\text{hol}}_T(k)$.  The former was shown in (\ref{hminusphi}) to equal
\begin{equation}\label{new616}
 h^{-}_T(r)   =  \frac{1}{8\pi^2} \int_0^{\infty} \mathcal{J}^-_r(x) \phi^{-}_T(x) \frac{dx}{x}
,
\end{equation}
where we recall from (\ref{ker}) and (\ref{KI}) that  $$ \mathcal{J}^-_r(x)    =   4\cosh(\pi r)K_{2ir}(x)  =  \pi i \frac{I_{2ir}(x)-I_{-2ir}(x)}{\sinh(\pi r)}.$$
If $r \leq 1$ and $x\leq 1$, the $I$-Bessel functions in the numerator are bounded by (\ref{Ibound}), while if $r\leq 1$ and $x>1$ the estimate \eqref{hm1} shows $\mathcal{J}^-_r(x)$ is bounded.  In either case  $\mathcal{J}^-_r(x) \ll 1/r$, so that
$$h^{-}_T(r)    \ll   \frac{1}{r} \int_0^{\infty} |\phi^{-}_T(x)|  \frac{dx}{x}    \ll  \frac{T}{r},$$
where we have used (\ref{inverse}) with $j=0$.

For $r\geq 1$ we apply a smooth partition of unity to (\ref{new616}) and write $h^{-}_T(r)=
{\mathcal I}_1(r)+{\mathcal I}_2(r)$, where
\begin{multline*}
{\mathcal I}_1(r)  :=  \int_0^{\infty} \mathcal{J}^-_r(x) \phi^{-}_T(x)w\left(\frac{x}{r^{1/3}}\right) \frac{dx}{x} \quad \text{and} \\
{\mathcal I}_2(r)  :=    \int_0^{\infty} \mathcal{J}^-_r(x) \phi^{-}_T(x)\left(1 - w\left(\frac{x}{r^{1/3}}\right)\right) \frac{dx}{x},
\end{multline*}
with $w$  a fixed,   smooth function on $\Bbb{R}$ such that  $w(x) \equiv 1$ on $[0, 1]$ and $w(x)\equiv 0$ on $[2,\infty)$.
We estimate ${\mathcal I}_2(r)$ trivially using  \eqref{inverse} and  \eqref{hm1}:
$${\mathcal I}_2(r) \ll   T \int_{r^{1/3}}^{\infty} (xT)^{-D/3} \left(\frac{r+x}{x}\right)^{1/10} \frac{dx}{x}     \ll    T^{1-\frac{D}{3}} r^{\frac{1}{15} - \frac{D}{9}}    \leq  Tr^{-5}$$
for $D \geq 50$. For ${\mathcal I}_1(r)$ we insert   the   power series expansion \eqref{power} to obtain
\begin{displaymath}
\begin{split}
{\mathcal I}_1(r) & \ll \sum_{k=0}^{\infty} \frac{1}{k!}\Bigl|\int_0^{\infty}  \frac{x^{\pm 2ir+2k}}{\Gamma(\pm 2ir +k+1)\sinh(\pi r)}\phi^{-}_T(x) w\left(\frac{x}{r^{1/3}}\right) \frac{dx}{x}\Bigr|\\
&  \ll    \sum_{k=0}^{\infty} \frac{1}{r^{k+1/2}k!}\Bigl|\int_0^{\infty}  x^{\pm2ir+2k}\phi^{-}_T(x) w\left(\frac{x}{r^{1/3}}\right) \frac{dx}{x}\Bigr|,\\
\end{split}
\end{displaymath}
where we have used $\Gamma(\pm2ir+k+1)=(\pm 2ir+k)\cdots (\pm 2ir+1)\Gamma(\pm 2ir+1)$ and then Stirling's formula to estimate the denominator.
We then apply integration by parts five times, integrating     $x^{\pm 2ir+2k-1}$ and differentiating
$ \phi_T^-(x)w(x/r^{1/3})$.  The latter gives   $$\sum_{i+j = 5} \left(\begin{array}{l} 5\\ j\end{array}\right) \left(\frac{d^{j}}{dx^{j}}\phi_T^-(x)\right)\frac{w^{(i)} (x/r^{1/3} )}{
r^{i/3}}.$$  Estimating trivially with (\ref{inverse}) (which is valid for $j\leq 5$ if $D>18$) and keeping in mind the integral is supported on $x\le 2r^{1/3}$, we find
 $${\mathcal I}_1(r) \ll
  \sum_{k=0}^{\infty} \frac{1}{r^{k+1/2}k!} \int_0^{2r^{1/3}} r^{-5}(2r^{1/3})^{2k}\max_{0 \leq j \leq 5} \Big|x^j \frac{d^j}{d x^j}\phi^{-}_T(x) \Big| \frac{dx}{x}   \ll
 \frac{T}{r^{5.5}}$$ This completes the analysis of $h^-_T(r)$.

We now turn to $h^{\text{hol}}(k)$, which by \eqref{hminusphi} and \eqref{ker1} can be written as
$$h^{\text{hol}}_T(k)   =  \frac{1}{8\pi^2} \int_0^{\infty} \mathcal{J}^{\text{hol}}_k(x) \phi^+_T(x) \frac{dx}{x}   =  \frac{i^k}{4\pi}\int_0^{\infty} J_{k-1}(x) \phi^+_T(x) \frac{dx}{x}.$$
We now employ the bound  \eqref{J-new} for $J_{k-1}(x)$ and the bound $\phi^+_T(x)\ll T\min(1,(xT)^{-D/3})$ (which follows from \eqref{inverse}), obtaining
$$
\aligned h^{\text{hol}}_T(k) &  \ll T \int_0^{(k-1)/4} \left(\frac{2x}{k-1}\right)^{k-1} \frac{dx}{x} + T\int_{(k-1)/4}^{\infty} (xT)^{-D/3} \frac{dx}{x} \\ & \ll T (2^{-k} + (kT)^{-D/3}),
\endaligned$$
uniformly in $k$.  This is clearly majorized by the claimed bound $Tk^{-5}$, completing the proof.
 \end{proof}

 Theorem \ref{thm1}   now follows as a consequence  of Lemma \ref{lem2}.

 \section{Nonvanishing and moments of $L$-functions}\label{sec6}

  In this section we prove Theorem \ref{thm2}, which essentially amounts to a proof of  the following lower bound.

\begin{lemma}\label{lem:lowerbound}
For $T$ sufficiently large we have
  \begin{equation*}
\frac{1}{2\pi} \int_{-\infty}^{\infty}  \frac{ |L(1/2 + it, \Pi)|^2}{|\zeta(1 + 2 it)|^2} h_T(t) dt    \gg   T \log T,
\end{equation*}
where $h_T$ is defined in {\rm (\ref{defH})}.
\end{lemma}
\begin{proof}
 We follow the method of Rudnick and Soundararajan \cite{RS}.
  Let $w$ be a fixed, non-negative, smooth function with support on $[1/4, 1]$ such that $w(t) \equiv 1$   on $[1/2, 3/4]$.   Like all compactly supported smooth functions, its Fourier transform satisfies
 \begin{equation}\label{rapiddecaywcheck}
   \check{w}(y)   =   \int_{-\infty}^{\infty} w(x)   e^{-ixy}   dx   \ll_B   (1+|y|)^{-B}
 \end{equation}
  for any $B > 0$. Let $T$ be a  large parameter  and
  let $A(t) $ be any continuous function that is not identically zero on $[\frac{1}{2} T, \frac{3}{4}T]$. Then the Cauchy-Schwarz inequality implies \begin{equation}\label{clearly}
\begin{split}
 \frac{1}{2\pi} \int_{-\infty}^{\infty}  \frac{ |L(1/2 + it, \Pi)|^2}{|\zeta(1 + 2 it)|^2} h_T(t) dt  &  \gg      \int_{T/4}^{T}  \frac{ |L(1/2 + it, \Pi)|^2}{|\zeta(1 + 2 it)|^2} w\left(\frac{t}{T}\right) dt  \\ & \geq   \frac{|I_1|^2}{ I_2},
\end{split}
\end{equation}
with
\begin{multline*}
I_1 := \int_{T/4}^T L(1/2 + i t, \Pi) A(t) w\left(\frac{t}{T}\right) dt \quad \text{and} \\  I_2 := \int_{T/4}^T |A(t)\zeta(1 + 2 it)|^2 w\left(\frac{t}{T}\right) dt.
\end{multline*}
  We choose
  \begin{equation}\label{Adef}
    A(t)   :=   \sum_{m} \frac{\overline{\lambda_{\Pi}(m)}}{m^{1/2 - it}}  w_1\left(\frac{m}{T^{1/8}}\right),
  \end{equation}
  where $w_1$ is a smooth, non-negative function with support on $[0, 1]$  such that $w_1(t) \equiv 1$ on $[0, 3/4]$,  and $\lambda_\Pi(m)$ is as defined in (\ref{delta1}). Its Mellin transform $\widehat{w}_1(s)$ is rapidly decaying and has a simple pole with residue 1 at $s=0$.
For $X > 1$ and $t \in \Bbb{R}$ we have
$$
\aligned
\sum_{n} \frac{\lambda_{\Pi}(n)}{n^{1/2 + it}} e^{-n/X} & = \int_{\Re{s}=1} L(1/2 + s + it, \Pi) \Gamma(s) X^s \frac{ds}{2\pi i} \\ & = L(1/2 + i t, \Pi) + O_\varepsilon(X^{-1/2} (1+|t|)^{2+\varepsilon})
\endaligned
$$
after shifting the contour to $\Re s = -1/2$.  We conclude
$$  I_1   =   \int_{T/4}^T  \Bigl( \sum_{n} \frac{\lambda_{\Pi}(n) e^{-n/T^8}}{n^{1/2 + it}}  + O_\varepsilon(T^{\varepsilon-2})\Bigr)\sum_{m  }
   \frac{\overline{\lambda_{\Pi}(m)}}{m^{1/2 - it}} w_1\!\left(\!\frac{m}{T^{1/8}}\!\right) w\!\left(\!\frac{t}{T}\!\right)\! dt.  $$
   The $m$-sum can be crudely bounded  by $T^{1/8}$ using (\ref{delta1}), so that integration over $t$ yields
   \begin{displaymath}
\begin{split}
  I_1    & =   T  \sum_{n, m} \check{w}\left(-T \log \frac{m}{n}\right) \frac{\lambda_{\Pi}(n)  \overline{\lambda_{\Pi}(m)}}{(nm)^{1/2}} w_1\left(\frac{m}{T^{1/8}}\right)e^{-n/T^8}   +    O(T^{-1/2})\\
   & =   T  \check{w}(0)  \sum_{n}  \frac{|\lambda_{\Pi}(n)|^2}{n}  w_1\left(\frac{n}{T^{1/8}}\right)   +   O(T^{-1/2}),
   \end{split}
   \end{displaymath}
where we used the Taylor expansion $e^{-n/T^8} = 1 + O(n/T^8)$ and that $$T \Big|\log\frac{m}{n}\Big|   \geq    T \log \left( 1 + \frac{1}{m}\right)   \gg   \frac{T}{m}   \gg    T^{7/8}$$   for $m\not= n\leq  T^{1/8}$, so that by (\ref{rapiddecaywcheck}) the off-diagonal contribution is negligible.   We conclude from Lemma \ref{lem0} and partial summation that
 \begin{equation}\label{I1}
\begin{split}
I_1 
    \gg    T\log T \end{split}
\end{equation}
for $T$ sufficiently large.   

On the other hand, for $\frac{1}4 T\leq t \leq T$ and any $Y > 1$ we have
\begin{equation*}
\aligned
  \sum_{n} \frac{1}{n^{1 + 2 i t}} w_1\left(\frac{n}{Y}\right) & = \int_{\Re s=2} \zeta(1 +s + 2 i t) \widehat{w}_1(s) Y^s \frac{ds}{2\pi i}\\ &  = \zeta(1 + 2it) + O_\varepsilon(Y^{-1}T^{1/2+\varepsilon}+T^{-100})
\endaligned
\end{equation*}
 upon shifting the contour to $\Re s = -1$, where the term $T^{-100}$ comes from the residue at $s=-2it$ and the rapid decay of the Mellin transform $\widehat{w}_1$.   Thus 
  \begin{equation*}\label{zeta12it2}
   \zeta(1+2it)^2  =  \Bigl(\sum_{n}n^{-1-2it}w_1\Bigl(\frac{n}{T^{2/3}}\Bigr)\Bigr)^2  +   O_\varepsilon(T^{-1/6+\varepsilon}),
\end{equation*}
since the $n$-sum is  $O(\log T)$.
Inserting this and (\ref{Adef})    into the definition of $I_2$ shows that
\begin{equation}\label{I2bound}
  I_2  = I_{2,\text{main}}   +  O_\varepsilon\left(\int_{T/4}^T|A(t)|^2T^{-1/6+\varepsilon}dt\right),
\end{equation}
where
 \begin{displaymath}
 \begin{split}
 I_{2,\text{main}}  = \int_{T/4}^T    \sum_{m_1, m_2}  & \frac{\overline{\lambda_{\Pi}(m_1)}\lambda_{\Pi}(m_2)}{m_1^{1/2 - it}m_2^{1/2+it}} w_1\left(\frac{m_1}{T^{1/8}}\right)w_1\left(\frac{m_2}{T^{1/8}}\right)      \times  
\\ & \sum_{n_1, n_2} \frac{w_1\left(\frac{n_1}{T^{2/3}}\right) w_1\left(\frac{n_2}{T^{2/3}}\right)}{n_1^{1+2it} n_2^{1 - 2it}}  w\left(\frac tT\right)dt\\
  =  T \sum_{\substack{m_1, m_2\\ n_1, n_2}} & \check{w}\left(T \log \frac{m_2n_1}{m_1n_2}\right) \frac{\overline{\lambda_{\Pi}(m_1)}\lambda_{\Pi}(m_2)}{(m_1m_2)^{1/2}n_1n_2 }    \   \times 
  \\ & w_1\left(\frac{m_1}{T^{1/8}}\right)w_1\left(\frac{m_2}{T^{1/8}}\right)w_1\left(\frac{n_1}{T^{2/3}}\right) w_1\left(\frac{n_2}{T^{2/3}}\right).  
\end{split}
\end{displaymath}
Since $A(t)=O(T^{1/16})$ (as follows from partial summation, Cauchy-Schwarz, and (\ref{RSbound})), the error term in (\ref{I2bound}) is $O_\varepsilon(T^{23/24+\varepsilon})=O(T^{24/25})$.

 The test function $w_1$ constrains $1\leq m_1,m_2\leq T^{1/8}$ and $1\leq n_1,n_2\leq T^{2/3}$.  For $m_1n_2 \not= m_2n_1$ we have   $ T|\log \frac{m_2n_1}{m_1n_2} |\gg T^{5/24}$, so that the off-diagonal contribution to $I_{2,\text{main}}$ is negligible because of (\ref{rapiddecaywcheck}). Every quadruple $(m_1, m_2, n_1, n_2)$ with $m_2n_1 = m_1 n_2$ must be of the form $(rs,rt,us,ut)$  for some integers $r, s, t, u$, hence enlarging the range of summation we have
 $$I_2   \ll    T\sum_{r, s, t, u \leq T} \frac{|\lambda_{\Pi}(rs)\lambda_{\Pi}(rt)|}{r (st)^{3/2} u^2}
 \ll
  T\sum_{r, s, t  \leq T} \frac{|\lambda_{\Pi}(rs)\lambda_{\Pi}(rt)|}{r (st)^{3/2} }
 . 
 $$
 Fix  $0<\varepsilon<\delta$. Using
 $$ \lambda_{\Pi}(rs) s^{-1/2 + \varepsilon}  \lambda_{\Pi}(rt) t^{-1/2 + \varepsilon}     \leq   |\lambda_{\Pi}(rs)|^2 s^{-1 + 2\varepsilon}   +  |\lambda_{\Pi}(rt)|^2 t^{-1 + 2\varepsilon},$$
we obtain
$$I_2  \ll   T \sum_{r, s \leq T} \frac{|\lambda_{\Pi}(rs)|^2}{rs^{2 - \varepsilon}}.$$
Write $d = \gcd(r, s)$, $r = r_0 d f$ and $s = s_0 d g$, with $\gcd(r_0s_0, d) = 1$ and $f, g\mid d^\infty$ (i.e., both $f$ and $g$ divide   some power of $d$). By the multiplicativity of $\lambda_{\Pi}$, we may factorize $\lambda_\Pi(rs)=\lambda_\Pi(r_0)\lambda_\Pi(s_0)\lambda_\Pi(d^2 fg)$.  Then
\begin{equation*}
\begin{split}
  I_2   \ll  &T
  \sum_{d\leq T}\sum_{\begin{smallmatrix}f,g\leq T \\ f,g\mid d^\infty\end{smallmatrix}}\sum_{r_0,s_0\leq T} \frac{|\lambda_\Pi(r_0)\lambda_\Pi(s_0)\lambda_\Pi(d^2 fg)|^2}{r_0 s_0^{2-\varepsilon} d^{3-\varepsilon} f g^{2-\varepsilon}}\\ \ll   &T(\log T)
  \sum_{d\leq T}\sum_{\begin{smallmatrix}f,g\leq T \\ f,g\mid d^\infty\end{smallmatrix}} \frac{| \lambda_\Pi(d^2 fg)|^2}{d^{3-\varepsilon} f g^{2-\varepsilon}}
\\     \ll  & T  (\log T) \sum_{d\leq T} \frac{1}{d^{1+4\delta-\varepsilon}}
  \sum_{f,g\mid d^\infty}\frac{1}{f^{2\delta}g^{1+2\delta-\varepsilon}},
  \end{split}
\end{equation*}
where we used (\ref{RSbound}) and partial summation to bound the $r_0, s_0$-sum and then applied (\ref{delta1}) in the last step. The $g$-sum is $O(1)$ since $\varepsilon<\delta$.  The $f$-sum is bounded by any small power of $d$, as can be seen by factoring it as a product over primes.  Thus the $d$-sum is $O(1)$ and   we  altogether obtain $I_2 \ll T \log T$.
 Together with \eqref{I1} we can therefore bound the left hand side of \eqref{clearly} from below by $\gg T \log T$ as desired. This completes the proof of the lemma.
\end{proof}

\begin{proof}[Proof of Theorem \ref{thm2}]
 Combining (\ref{defM}), Theorem~\ref{thm1}, and Lemma~\ref{lem:lowerbound}, we have
\begin{displaymath}
\begin{split}
 - \sum_{\pi} \epsilon_{\pi} &  \frac{L(1/2,   \Pi \times \pi ) }{L(1, \text{{\rm Sym}}^2 \pi)} h_T(t_{\pi}) \\ & =    \frac{1}{2\pi} \int_{-\infty}^{\infty}  \frac{|L(1/2 + it, \Pi)|^2}{|\zeta(1 + 2 it)|^2} h(t) dt   -\mathcal{M}^-(h_T)  \gg T\log T
\end{split}
\end{displaymath}
for $T$ sufficiently large. Fix $\varepsilon > 0$ and choose $D=50+\varepsilon^{-1}$ in \eqref{defH}. By \eqref{convex} the contribution of representations $\pi$ with $t_{\pi} \leq T^{1-\varepsilon}$ or $t_{\pi} \geq  T^{1+\varepsilon}$ is $O_{\Pi, \varepsilon}(1)$, hence there must be a non-zero $L$-value in the interval $T^{1-\varepsilon} \leq t_{\pi} \leq T^{1+\varepsilon}$, provided $T$ is sufficiently large in terms of  $\Pi$ and $\varepsilon$. Hence the number of non-vanishing $L$-values with $t_{\pi} \leq X$ is at least $c_1 \varepsilon^{-1} (\log\log X - c_2)$ for some absolute constant $c_1 > 0$ and some constant $c_2  = c_2(\Pi, \varepsilon) > 0$, proving Theorem \ref{thm2}.
\end{proof}


\appendix

 \section{Bessel functions}\label{bessel}

 In this section we compile some useful facts about the $I$-, $J$-, and $K$-Bessel functions  for easy reference. We have \cite[8.486.2 and 11]{GR}
 \begin{equation*}
    K'_{\nu}(x)  = - \frac{1}{2}(K_{\nu+1}(x) + K_{\nu-1}(x)), \quad I'_{\nu}(x)  =   \frac{1}{2}(I_{\nu+1}(x) + I_{\nu-1}(x))
 \end{equation*}
 for $x > 0$ and $\nu \in \Bbb{C}$, and hence
 \begin{equation}\label{diff}
    K^{(j)}_{2it}(x)   =   \left(-\frac{1}{2}\right)^j \sum_{n=0}^j \left(\begin{array}{c} j\\ n\end{array}\right) K_{2it - j + 2n}(x)
 \end{equation}
  and
  \begin{equation}\label{diff1}
    I^{(j)}_{2it}(x)  =   \left(\frac{1}{2}\right)^j \sum_{n=0}^j \left(\begin{array}{c} j\\ n\end{array}\right) I_{2it - j + 2n}(x)
 \end{equation}
for $j \in \Bbb{Z}_{\geq 0}$.   A weak version of \cite[Proposition 9]{HM} states that
 \begin{equation}\label{hm1}
  {\mathcal J}_t^-(x) = 4 \cosh(\pi t) K_{2it}(x)    \ll_{\Im t}    e^{\min(0, -x + \pi |\Re t|)} \left(  \frac{1 + |t| + x}{x}\right)^{2|\Im t| + 1/10}
 \end{equation}
 for $x > 0$ and $t \in \Bbb{C}$. The Bessel $K$-function and the Bessel $I$-function are related \cite[8.485]{GR} via
 \begin{equation}\label{KI}
  \sinh(\pi t)K_{2it}(x)    =  \frac{\pi i}{4} \frac{I_{2it}(x) - I_{-2it}(x)}{\cosh(\pi t)}.
 \end{equation}
 It follows from the power series expansion \cite[8.445]{GR} and Stirling's formula that
 \begin{equation}\label{power}
 I_{\nu}(x)  =  \sum_{k=0}^{\infty} \frac{1}{k! \Gamma(\nu + k + 1)} \left(\frac{x}{2}\right)^{\nu+2k}
 \end{equation}
    that
 \begin{equation}\label{Ibound}
   e^{-\pi |t|}  I_{2it}(x)     \ll_{\Im t}   \frac{x^{-2\Im t}}{(1+|t|)^{1/2 - 2\Im t}}
 \end{equation}
 for $t \in \Bbb{C}$ and $0 < x < (1 + |t|)^{1/2}$ (the estimate is trivial when $t$ is constrained to a compact region; for $\Re t$ large the bound guarantees the terms in (\ref{power}) are bounded by a constant multiple of the $k=0$ term).

 The Bessel kernels defined in \eqref{ker} and \eqref{ker1} have   Mellin transforms
 \begin{equation}\label{kerhat}
 \begin{split}
  \widehat{\mathcal{J}}^{+}_t(s)    & =   2^{s}    \Gamma(s/2 + it) \Gamma(s/2 - it)  \cos(\pi s/2),\\
   \widehat{\mathcal{J}}^{-}_t(s)    & =     2^{s}  \Gamma(s/2 + it) \Gamma(s/2 - it) \cosh(\pi t), \ \text{and} \\
\widehat{\mathcal{J}}^{\text{hol}}_k(s)    & =    
2^s \cos\left(\frac{\pi s}{2}\right) \Gamma\left(\frac{s+1-k}{2}\right)\Gamma\left(\frac{s+k-1}{2}\right)\\ & = \frac{i^k 2^s \pi \Gamma(\frac{1}{2}(s+k-1))}{\Gamma(\frac{1}{2}(1+k-s))},
\end{split}
\end{equation}
where $k\geq 2$ is even.
These follow from \cite[17.43.16 \& 18]{GR} together with functional equation of the Gamma function.

If $\phi$ is any even Schwartz  function bounded by a constant multiple of $e^{-2\pi |t|}$, then
\begin{equation}\label{four}
\int_{-\infty}^\infty {\mathcal J}_t^{-}(x)\phi(t)dt   = 2\int_{-\infty}^\infty e^{i x\sinh(v)}\check{\phi}(2v)\, dv,
 \end{equation}
where $\check{\phi}(v) = \int_{-\infty}^\infty \phi(x)e^{-ivx}dx$.  This is Parseval's equality, with the analytic subtlety that the Fourier transform of the tempered distribution  ${\mathcal J}_t^{-}(x)$ given in  \cite[8.432.4]{GR} is not integrable.  To prove this identity, first assume $\phi$ has compact support and insert the absolutely convergent integral formula $K_{2it}(x)=\frac{1}{2}\int_{-\infty}^\infty e^{-x\cosh(u)-2itu}du$ \cite[8.432.1]{GR} into
definition  (\ref{ker}) to rewrite the left hand side of (\ref{four}) as
$$
2\int_{-\infty}^\infty \int_{-\infty}^\infty  e^{\pi   t}e^{-x\cosh(u)-2itu}\phi(t) \, dt\, du.
$$  This double integral is also absolutely convergent, with an integrand that decays doubly-exponentially in $u$ uniformly in horizontal strips.  Shifting the $u$-contour down to $\Im{u}=-\pi/2$ and changing the order of integration establishes (\ref{four}).   (Note that the bound
(\ref{combine1}) assures the $O(e^{-2\pi |t|})$ condition, so that (\ref{four}) applies to (\ref{v-int1})).

 Finally we state the   uniform bound
\begin{equation}\label{J-new}
   J_{k}(x) \ll \min\left((2x/k)^k, 1 \right)
\end{equation}
for $k \in \Bbb{N}$. The first bound    follows from \cite[Lemma 4.1]{Ra} for $k \geq 15$ and from the power series expansion for $1 \leq k \leq 14$ and $x \leq 1$, while the second bound follows from   \cite[8.411.1]{GR}.

\end{document}